\documentclass{amsart}
\usepackage[noadjust]{cite}
\usepackage{todonotes}
\newcommand{\ol}{\overline}
\newcommand{\NN}{\mathbb{N}}
\newcommand{\CC}{\mathbb{C}}
\newcommand{\DD}{\mathbb{D}}
\newcommand{\RR}{\mathbb{R}}

\newcommand{\LL}{\mathbb{L}}
\newcommand{\cM}{\mathcal{M}}
\newcommand{\cP}{\mathcal{P}}
\newcommand{\cQ}{\mathcal{Q}}

\newcommand{\wh}{\widehat}

\usepackage{hyperref}

\usepackage{thmtools}
\usepackage[nameinlink]{cleveref}

\usepackage{lipsum}

\usepackage{amsmath, amsthm, amssymb}
\usepackage{fancyhdr}

\usepackage{tikz}
\usepackage{tikz-cd}
\usepackage{tikz-3dplot}
\usepackage{quiver}
\usepackage{eucal}[mathcal]
\usepackage{enumitem}

\theoremstyle{plain}
\newtheorem{theorem}{Theorem}[section]
\newtheorem{corollary}[theorem]{Corollary}
\newtheorem{proposition}[theorem]{Proposition}
\newtheorem{lemma}[theorem]{Lemma}
\newtheorem*{thm*}{Theorem}

\theoremstyle{definition}

\newtheorem{remark}[theorem]{Remark}
\newtheorem{definition}[theorem]{Definition}
\newtheorem{notation}[theorem]{Notation}

\newcommand{\Cat}{\mathbf{Cat}}

\newcommand{\op}{\mathrm{op}}

\newcommand{\sym}{\mathrm{sym}}
\newcommand{\cW}{\mathcal W}

\newcommand{\Fun}{\mathrm{Fun}}
\newcommand{\Hom}{\mathrm{Hom}}

\newcommand{\Seq}{{\mathbb S\mathrm{eq}}}

\DeclareMathOperator{\ob}{Ob}
\DeclareMathOperator{\dom}{dom}

\DeclareMathOperator{\Ob}{Ob}

\newcommand{\cC}{\mathcal C}
\newcommand{\dD}{\mathcal D}

\newcommand{\RCat}{{\mathbb R_+\text{-}\Cat}}
\newcommand{\cA}{\mathcal{A}}
\newcommand{\cV}{\mathcal{V}}
\newcommand{\cY}{\mathcal Y}
\newcommand{\cL}{\mathcal L}

\newcommand{\cB}{\mathcal B}
\newcommand{\cD}{\mathcal D}

\newcommand{\bR}{\mathbb{R}}
\newcommand{\bA}{\mathbb{A}}
\newcommand{\bB}{\mathbb{B}}
\newcommand{\bC}{\mathbb{C}}
\newcommand{\bD}{\mathbb{D}}

\newcommand{\bI}{\mathbb{I}}
\newcommand{\bJ}{\mathbb{J}}
\newcommand{\bL}{\mathbb{L}}
\newcommand{\Set}{\mathbf{Set}}
\newcommand{\Idem}{\mathbf{Idem}}
\newcommand{\Split}{\mathbf{Split}}
\newcommand{\Cof}{\mathcal C\mathrm{of}}

\newcommand{\Fib}{\mathcal F\mathrm{ib}}
\newcommand{\Inj}{\mathcal I\mathrm{nj}}
\newcommand{\Iso}{\mathcal I\mathrm{so}}
\newcommand{\Surj}{\mathcal S\mathrm{urj}}
\newcommand{\IdFib}{\mathcal I\mathrm d\mathcal F\mathrm{ib}}
\newcommand{\sseq}{\subseteq}
\newcommand{\spseq}{\supseteq}
\newcommand{\wt}{\widetilde}

\newcommand{\id}{\mathrm{id}}

\newcommand{\bN}{\mathbb{N}}
\newcommand{\into}{\hookrightarrow}

\renewcommand{\RR}{\mathbb{R}}
\newcommand{\vare}{\varepsilon}

\makeatletter
\newcommand{\skipitems}[1]{%
	\addtocounter{\@enumctr}{#1}%
}
\makeatother

\hypersetup{%
	colorlinks,%
	linkcolor={red!60!black},%
	citecolor={red!60!black},%
	urlcolor={red!60!black}%
}

\def\on{\operatorname}
\def\scr{\EuScript}
\def\bb{\mathbb}
\def\sf{\mathsf}

\title[Homotopical models for completeness]
 {Homotopical models for metric spaces and completeness}
\author[Dailey]{Isaiah Dailey}
\email{isaiah@ucsd.edu}
\author[Huggins]{Clara Huggins}
\email{clararh@umich.edu}
\author[Mujevic]{Semir Mujevic}
\email{semir\_mujevic@brown.edu}
\author[Shupe]{Chloe Shupe}
\email{cshupe@brynmawr.edu}

\address{%
}
\subjclass{18D20, 18N40}
\keywords{Lawvere metric spaces, model categories, Cauchy completion, enriched categories}
\date{\today}
\dedicatory{Last Revised: \today}

\begin{document}
\begin{abstract}
Categories enriched in the opposite poset of non-negative reals can be viewed
as generalizations of metric spaces, known as Lawvere metric spaces. In this
article, we develop model structures on the categories $\RCat$ and $\RCat^\sym$
of Lawvere metric spaces and symmetric Lawvere metric spaces, each of which
captures different features pertinent to the study of metric spaces. More
precisely, in the three model structures we construct, the fibrant-cofibrant
objects are the extended metric spaces (in the usual sense), the Cauchy
compl/ete Lawvere metric spaces, and the Cauchy complete extended metric spaces,
respectively. Finally, we show that two of these model structures are unique in
a similar way to the canonical model structure on $\Cat$.
\end{abstract}
\maketitle
\tableofcontents

\section*{Introduction}

An unfortunate feature of most conventional categories of metric spaces is that they do not have particularly nice categorical properties. For instance, not all limits and colimits exist in categories of metric spaces. One way to alleviate these difficulties is to pass to a less restrictive notion of metric spaces. In his 1973 paper \cite{Lawvere}, Lawvere noted that the framework of enriched category theory provides a particularly convenient way to do just this. If one considers the interval $[0,\infty]$ as a sub-poset of the extended reals, and equips the opposite category with the (closed) monoidal structure given by addition, one obtains a category---herein denoted $\RR_+$---such that $\RR_+$-enriched categories are relaxed analogs of extended metric spaces. The resulting axioms require that the ``distance" from every point to itself is $0$ and that the ``metric'' satisfies the triangle inequality, but nothing more. 

Within this analogy, the enriched functors between $\RR_+$-enriched categories
correspond to Lipschitz maps of coefficient 1, also called \emph{short
maps}. The  corresponding category $\RR_+\text-\Cat$ of $\RR_+$-enriched
categories is substantially better suited to the application of categorical
techniques than is the corresponding subcategory of metric spaces. Categories
enriched over $\RR_+$---also sometimes called \emph{Lawvere metric
spaces}---are of interest in applied category theory (see,
e.g., \cite[pg. 60]{Fritz} or \cite[\S 2.3.3]{SpivakFongSketches}). However, the degree to which they offer a chance to apply categorical techniques to metric
analysis has been little explored outside of Lawvere's original work
\cite{Lawvere} and the subsequent paper \cite{BorceuxDejean}.

Both of these works concern themselves with the subject of \emph{Cauchy completions} of categories. Per \cite[\S 4]{BorceuxDejean}, given a complete and cocomplete symmetric monoidal closed category $\scr{V}$, the \emph{Cauchy completion} of a $\scr{V}$-enriched category $\scr{C}$ is the full $\scr{V}$-subcategory $\overline{\scr{C}}$ of $\scr{V}$-profunctors from the monoidal unit $I$ to $\scr{C}$ on those profunctors that have right $\scr{V}$-adjoints as described in \emph{op.\ cit}. Following \cite{Lawvere}, we refer to these right adjoints as \textit{duals}. There are a wide variety of equivalent characterizations of $\overline{\scr{C}}$, but for our purposes the key fact about $\overline{\scr{C}}$ lives in the $\RR_+$-enriched setting: If we begin with an $\RR_+$-category $\CC$ that corresponds to a genuine \emph{metric space}, then the Cauchy completion of $\CC$ as an $\RR_+$-category coincides with the completion of $\CC$ as a metric space.

The Cauchy completion is quite well-behaved categorically, mainly stemming from
its relation to the enriched Yoneda embedding. Of particular interest is
the fact that a functor $f:\CC\to \DD$ between $\RR_+$-categories induces a
functor $\overline{f}:\overline{\CC}\to \overline{\DD}$ between Cauchy
completions---effectively by left Kan extension. This construction is
pseudo-functorial, and behaves similarly to a concept from abstract homotopy
theory: that of a \emph{functorial fibrant replacement}. This latter notion is
a key part of the theory of \emph{model categories}, introduced by Quillen in
\cite[Ch. 1]{Quillen}.

\subsection*{The model structures}

Our aim in this article is to develop model categories that give insight into
the theory of various kinds of metric spaces specifically extended and Cauchy
complete metric spaces. To this end, we construct model structures on the two
closely related categories $\RR_+\text-\Cat$ and
$\RR_+\text-\Cat^{\on{sym}}$. The former is simply the category of (small)
$\RR_+$-enriched categories; the latter is the full subcategory on those
$\RR_+$-categories that are \emph{symmetric}, meaning that the hom-object (an
element of $\RR_+$) from an object $x$ to an object $y$ is the same as the
hom-object from $y$ to $x$. The subcategory $\RR_+\text-\Cat^{\on{sym}}$ is
closed under the formation of small limits and colimits, so no great difficulty
occurs in passing arguments between these two settings. These model structures
are constructed so that their fibrant and cofibrant objects are precisely the
extended and Cauchy complete metric spaces, similar to Mihara's study of
ultrametric spaces in \cite{Mihara}.

Our first main result is to establish that there is a \emph{unique} model
structure on each of these categories whose weak equivalences are the
$\bR_+$-enriched equivalences of categories and whose fibrant and cofibrant
objects are gaunt, meaning they have the property that two objects distance
zero apart are necessarily identical. In the case of
$\RR_+\text-\Cat^{\on{sym}}$, this means that the homotopy theory modeled by
the model structure is, in fact, the category of extended metric spaces with
short maps as morphisms. More precisely, we prove the following.

\begin{theorem}[The Metric model structure]
	There is a unique model structure on $\RR_+\text-\Cat$ (resp.\ $\RR_+\text-\Cat^{\on{sym}}$) such that
	\begin{enumerate}
		\item The weak equivalences are the fully faithful and essentially surjective $\RR_+$-functors, i.e., the equivalences of $\bR_+$-categories.
		\item Not every object is both fibrant and cofibrant.
	\end{enumerate}
	Moreover, the fibrant-cofibrant objects in this model structure are precisely the gaunt $\bR_+$-categories (resp.\ the symmetric gaunt $\bR_+$-categories).
\end{theorem}

The uniqueness result appears in the text as \autoref{thm:metric_MS_unique}, and existence is proven in \autoref{theorem:IoI_model_structure}, along with a description of the fibrations and cofibrations in this model structure. The other model structures in this article deal with Cauchy completeness in various ways. Because of the technical difficulties introduced by considering Cauchy sequences in the non-symmetric setting, we restrict our constructions to $\RR_+\text-\Cat^{\on{sym}}$. We first prove the following theorem.

\begin{theorem}[The Cauchy Model Structure]
  There is a model structure on $\RR_+\text-\Cat^{\on{sym}}$ such that 
	\begin{enumerate}
    \item The weak equivalences are the fully faithful and dense
    $\RR_+$-functors.
    \item The fibrant-cofibrant objects are the Cauchy-complete and symmetric
    $\RR_+$-categories.
	\end{enumerate}
\end{theorem}

This theorem is listed in the text as
\autoref{theorem:cauchy_model_structure}. Finally, combining the arguments used
to prove the existence of these two model structures, we construct a third
intermediate model structure.

\begin{theorem}[The Cauchy-Metric Model Structure]
	There is a unique model structure on $\RR_+\text-\Cat^{\on{sym}}$ such that 
	\begin{enumerate}
    \item The weak equivalences are the fully faithful and dense $\RR_+$-functors. 
		\item Every fibrant-cofibrant object is gaunt.
	\end{enumerate}
\end{theorem}

The existence of this model structure is proven in the text as
\autoref{theorem:cauchy_metric_model_structure}, and the uniqueness result is
proven in \autoref{lemma:cauchy_metric_uniqueness}. In each of our proofs, we
provide an explicit construction of the model structure and verification of the
axioms, to develop the detail that could become useful in later work. One
possible future direction of significant interest would be the development of
related model structures in more general enriched settings, perhaps for
categories enriched in quantales. 

\subsection*{Structure of the paper}
We begin by providing some background and preliminaries in
\Cref{section:preliminaries}, including the basics of $\RR_+$-categories and
Cauchy completions of $1$-categories and $\bR_+$-categories. In
\Cref{section:Karoubian}, we prove the existence of a model structure on $\Cat$
(the category of small categories). This model structure characterizes the
idempotent completion of categories, and we call it the \textit{Karoubian
model structure}, following \cite[\S2]{JcatLabModels}. The same model structure
was constructed in \cite[\S 1]{MoritaHmtpy} as a Bousfield
localization of the canonical model structure on $\Cat$, where it is called the
\emph{Morita model structure}. We provide an alternative proof for the
existence of this model structure; this proof additionally serves as a
``warm-up'' for our later work, as the Karoubian model structure closely
parallels our construction of the Cauchy model structure on $\RCat^\sym$.

We then turn our attention to $\RR_+$-categories, proving the three theorems
stated above. Each of the next three sections is devoted to the construction of
a model structure and verification of the concomitant axioms:
\Cref{section:metric_model_structure} for the structure that models extended
metric spaces, \Cref{section:cauchy_model_structure} for the structure that
models Cauchy-complete symmetric $\RR_+$-categories, and
\Cref{section:cauchy_metric_model_structure} for the structure that models
Cauchy-complete extended metric spaces. A number of analytically-flavored
results, mostly variants of well-known statements about metric spaces proven
for $\mathbb{R}_+$-categories, are collected in an appendix and referred to
throughout the article.

\section{Cauchy completion}\label{section:preliminaries}

This section will discuss the Cauchy completion, as described in \cite{Lawvere}, and later expanded by Borceux and Dejean in \cite{BorceuxDejean}. The first subsection, discusses the Cauchy completion of categories as described in \S1--2 of \textit{op.\ cit.} We define the first class of weak equivalences we will consider---termed \textit{pastoral equivalences}---as well as a couple characterizations of the definition. We will then use these functors to discuss the Cauchy completion of categories. This work will be important for our construction of the Karoubian model structure on $\Cat$ in \Cref{section:Karoubian}: The weak equivalences in this model structure are precisely the pastoral equivalences, and the fibrant-cofibrant objects are the Cauchy-complete categories.

In the second subsection, we discuss the theory of Cauchy completion in enriched category theory. This article will not consider these concepts in full generality, instead restricting our attention to categories enriched over a specific poset $\RR_+$. These ``$\bR_+$-categories'' will provide the analogs of metric spaces mentioned in the introduction. Along the way, we will draw comparisons between Cauchy completion in the set-enriched and $\bR_+$-enriched settings. This subsection, together with \autoref{app:analysis}, will recapitulate and further develop the framework of $\RR_+$-enriched categories of \cite{Lawvere,BorceuxDejean}. The appendix in particular is devoted to the proofs of generalizations of properties of limits in metric spaces to the setting of $\RR_+$-categories. While we will not consider the general enriched setting, we expect that some of the results in this and the following sections will generalize beyond $\RR_+$-categories.

\begin{remark}\label{rmk:small}
    All categories are assumed to be small unless stated otherwise. Note that the categories $\Set$, $\Cat$, $\RR_+\mathbf{-Cat}$, etc. are not small.
\end{remark}

\subsection{Cauchy completion of categories}

First, we give some important definitions. 

\begin{definition}
	Given a category $\cC$, let $\Set_\cC:= \text{Fun}(\cC^\op, \Set)$ denote the \textit{presheaf category} i.e., the category of functors from $\cC^\op$ to $\Set$.
\end{definition}

\begin{definition}
	We say that $F:\cC \to \cD$ is a \textit{pastoral equivalence} if the post-composition functor $F^*:\Set_\cD\to\Set_\cC$ is an equivalence of categories.
\end{definition}

Now, we prove an alternative characterization of pastoral equivalences. We
first introduce the categories $\Idem$ and $\Split$ to simplify work with split
idempotents. Moreover, we define Cauchy complete categories as categories where
all idempotents split.

\begin{definition}
    Define the category $\Idem$ to have a single object $0$ and a single non-identity morphism $e:0\to0$ with $e\circ e=e$. Define the category $\Split$ to be the category with two objects $0$ and $1$ freely generated by two arrows $p:0\to 1$ and $q:1\to 0$ such that $p\circ q=\id_1$. Let $\Sigma:\Idem\to\Split$ be the ``inclusion functor'' that takes $e$ to $q \circ p$. The functor $\Sigma$ is fully faithful.
\end{definition}

\begin{remark}
	Given a category $\cC$, a functor $F:\Idem\to\cC$ is precisely an idempotent in $\cC$. A functor $F:\Split\to\cC$ is a split idempotent in $\cC$, equivalently, a retract in $\cC$.
\end{remark}

The following definition is similar in form to the definition of surjective up to isomorphisms from the canonical model structure on $\Cat$ \cite{RezkCat} and behaves in many of the same ways.

\begin{definition}\label{def:surjective_up_to_retracts}
	Given a functor $F:\cC\to\cD$, we say that $F$ is \textit{surjective up to retracts} if every object in $\cD$ is a retract of an object in the image of $F$.
\end{definition}

We proceed to define Cauchy completions and their interactions with idempotents following \cite{BorceuxDejean} and \cite{Lawvere}.

\begin{definition}\label{def:cauchy_completion_category}
	Let $\cC$ be a category. We define its \textit{Cauchy completion}, denoted by $\overline\cC$, to be the full subcategory of $\Set_\cC$ on the objects that are retracts of objects in the image of the Yoneda Embedding $\cY:\cC\to\Set_\cC$.
\end{definition}

\begin{lemma}\label{lemma:Cauchy_completion_is_Cauchy_complete}
	Given a category $\cC$, every idempotent splits in $\overline\cC$. In other words, every functor $\Idem\to\overline\cC$ factors through a functor $\Split\to\overline\cC$ via $\Sigma$.
\end{lemma}
\begin{proof}
	This is part 2 of \cite[Theorem 1]{BorceuxDejean}.
\end{proof}

\begin{remark}
	The Yoneda embedding determines a fully faithful inclusion functor $\iota_\cC:\cC\into\overline\cC$ that is surjective up to retracts.
\end{remark}

We say a category $\cC$ is \textit{Cauchy complete} if $\iota_\cC$ is an equivalence of categories, equivalently, if every idempotent splits in $\cC$.

One may ask why the category $\ol\cC$ constructed in \autoref{def:cauchy_completion_category} deserves to be called the ``Cauchy completion'' of $\cC$. First, note that one characterization of $\ol\cC$ is as the completion of $\cC$ under absolute colimits\footnote{A colimit diagram is called absolute when it is preserved by every functor.} (see the discussion before \cite[Proposition 3]{BorceuxDejean}, or the discussion after the proof of \cite[Theorem 5.36]{kelly}). 
As a consequence of this characterization of $\overline{\mathcal{C}}$, we have
the following result, that tells us we may uniquely extend a functor
between two categories to a functor between their Cauchy completions.

\begin{proposition}\label{proposition:Induced_functor_Cauchy_completion}
	Let $F:\cC\to\cD$ be a functor. There exists a functor $\overline F:\overline\cC\to\overline\cD$ such that the following diagram commutes. 
	\[\begin{tikzcd}
		\cC & \cD \\
		\ol\cC & \ol\cD
		\arrow["F", from=1-1, to=1-2]
		\arrow["{\iota_\cC}"', from=1-1, to=2-1]
		\arrow["{\ol F}", from=2-1, to=2-2]
		\arrow["{\iota_\cD}", from=1-2, to=2-2]
	\end{tikzcd}\]
	Moreover, the functor $\overline F$ is unique up to unique natural isomorphism.
\end{proposition}

As to why $\overline{\mathcal{C}}$ deserves to be called the \textit{Cauchy}
completion of $\cC$, we defer this discussion to the following subsection,
after we have developed the theory of $\bR_+$-enriched categories. To conclude
this subsection, we characterize the relationship between surjectivity up to
retracts, the Cauchy completion of categories, and pastoral equivalence. While
we expect the following statement is not novel, we have been unable to find a
citation and thus include the proof for the sake of completeness.

\begin{proposition}\label{proposition:surj_retract_equivalence}
	Let $F: \cC\to\dD$ be a functor between categories. Then the following are equivalent. 
	\begin{enumerate}
		\item The functor $F$ is fully faithful and surjective up to retracts.
		\item The induced functor $\overline F: \overline\cC\to\overline\dD$ on Cauchy completions is an equivalence of categories.
		\item $F$ is a pastoral equivalence.
	\end{enumerate}
\end{proposition}
\begin{proof}
  (1)$\implies$(2). Let $d$ be an object in $\overline{\scr{D}}$. Since
  $\iota_{\scr{D}}\circ F$ is surjective up to retracts and fully faithful,
  there is an idempotent $e_d:c_d\to c_d$ in $\scr{C}$ such that
  $\iota_{\scr{D}}(F(e_d))$ splits through the object $d$. By
  \autoref{lemma:Cauchy_completion_is_Cauchy_complete}, we may choose a
  splitting of $\iota_{\scr{C}}(e_d)$ through an object $b_d$ in
  $\overline{\scr{C}}$, thus, the idempotent
  $\overline{F}(\iota_{\scr{C}}(e_d))=\iota_{\scr{D}}(F(e_d))$ splits through
  both $d$ and $\overline{F}(b_d)$. It follows $\overline{F}(b_d)\cong d$, so
  $\overline{F}$ is essentially surjective. Moreover, since $\ol
  F\circ\iota_\cC=\iota_\cD\circ F$ and $\iota_\cC$ are fully faithful, $\ol F$
  must be as well.

	(2)$\implies$(1). As a composition of functors that are fully faithful and surjective up to retracts, $\iota_\cD\circ F=\ol F\circ\iota_\cC$ is likewise fully faithful and surjective up to retracts. Since $\iota_\cD$ and $\iota_\cD\circ F$ are fully faithful, $F$ must be as well. It remains to show $F$ is surjective up to retracts. Let $d$ be an object in $\cD$. We know $\iota_\cD\circ F$ is surjective up to retracts, so that $\iota_\cD(d)$ is a retract of some object $\iota_\cD(F(c))$ in the image of $\iota_\cD\circ F$. Fully faithful functors reflect retracts in their image so $d$ is necessarily a retract of $F(c)$ in $\cD$.
	
	(2)$\iff$(3). This follows from  applying the 2-functorial assignment $\Cat^\op\to\Cat$ sending $\cC\mapsto\Set_\cC$ and $F\mapsto F^*$ to the defining commutative diagram for $\overline{F}$ in 
	\autoref{proposition:Induced_functor_Cauchy_completion} together with \cite[Thm 1, (4)]{BorceuxDejean}.
 
\end{proof}

\subsection{Cauchy completion of \texorpdfstring{$\mathbb{R}_+$}{R+}-categories}

Now we proceed to our discussion of $\bR_+$-enriched categories and their
Cauchy completions. While we will not make heavy use of the theory of enriched
categories in this article, the interested reader can find a quite accessible
treatment in \cite{kelly}. Intuitively, given a monoidal category $\cV$ (that need not be small), a
$\cV$-enriched category $C$ should be thought of as a category that has
\textit{hom-objects} $\cV\text-\Hom_C(x,y)$ from $x$ to $y$ living in $\cV$,
rather than hom-sets. Furthermore, composition in $C$ is expressed via a
morphism 
\[
  \cV\text-\Hom_C(y,z)\otimes\cV\text-\Hom_C(x,y)\to\cV\text-\Hom_C(x,z)
\]
in $\cV$, where $-\otimes-:\cV\times\cV\to\cV$ is the monoidal product functor
in $\cV$. For example, from the enriched point of view, a (locally small)
category is simply a $\Set$-enriched category, a pre-additive category is an
$\mathbf{Ab}$-enriched category, and a (strict) $2$-category is a
$\Cat$-enriched category. Now we may define the enriching category that will be
important to us.

\begin{definition}
	The category $\bR_+$ is the opposite category associated with the partially ordered set $[0,\infty]$. That is, there is a unique morphism $a\to b$ in $\bR_+$ if and only if $a\geq b$.
\end{definition}

There is a product operation on this category: a functor \[\begin{tikzcd}[row sep=0em]
	(-)+(-): &[-3em] \bR_+\times\bR_+\arrow[r] & \bR_+ \\
	& (a,b) \arrow[r,mapsto] & a+b,
\end{tikzcd}\]
where we adopt the convention that $a+\infty=\infty+a=\infty$ for all $a\in\bR_+$. This product has a neutral element: the object $0$ in $\bR_+$. This is also the terminal object in $\bR_+$. Moreover, if we adopt the convention that $\infty-\infty=0$, $a-\infty=-\infty$, and $\infty-a=\infty$ for all $a\in(0,\infty)$, then we can define a functor
\[\begin{tikzcd}[row sep=0em]
	[-,-]: &[-3em] \bR_+\times\bR_+\arrow[r] & \bR_+\\
	& (a,b)\arrow[r, mapsto] & \text{max}(b-a,0),
\end{tikzcd}\]
so that there is a natural isomorphism
\[\Hom_{\bR_+}(a+b,c)\cong\Hom_{\bR_+}(a,[b,c]).\]

Note that $\bR_+$ is not closed under subtraction with the conventions given above. Furthermore addition does not ``associate'' with subtraction: It is not true that $(a+b)-c=a+(b-c)$ in $\bR_+$ (consider $a=b=c=\infty$). Despite these oddities, our conventions are justified by the fact that $[-,-]$ gives $\bR_+$ the structure of a (strict!) symmetric monoidal closed category.

\begin{definition}
	A category enriched over the monoidal category $(\RR_+,+,0)$ is called an \emph{$\RR_+$-enriched category}, or an \emph{$\RR_+$-category}. An enriched functor of $\RR_+$-enriched categories is an \emph{$\RR_+$-functor}.
\end{definition}

In accordance with \autoref{rmk:small}, we will only consider $\bR_+$-categories that are \textit{small}, i.e., those
$\bR_+$-categories $\bC$ whose collection of objects forms a set rather than a
proper class. For the definitions of enriched categories and enriched functors,
we refer the reader to \cite[\S1.2]{kelly}. In this case, an
$\bR_+$-category $\bC$ is the data of a set of objects $\ob\bC$, along with,
for each pair of objects $x$ and $y$ in $\Ob\bC$, an object $\bC(x,y)$ in
$\bR_+$, called the ``hom-object from $x$ to $y$''. In addition, for each
triple of objects $x$, $y$, and $z$ in $\bC$, we require that there exists a
morphism in $\bR_+$:
\[
  \bC(y,z)+\bC(x,y)\to\bC(x,z)
\]
i.e., we require that $\bC(y,z)+\bC(x,y)\geq\bC(x,z)$. Similarly, an
$\bR_+$-functor $f:\bC\to\bD$ is a function $f:\ob\bC\to\ob\bD$, along with,
for each pair of objects $x$ and $y$ in $\bC$, a morphism
 \[
   f_{x,y}:\bC(x,y)\to\bD(f(x),f(y))
 \]
in $\mathbb{R}_+$, i.e., we require that $\bC(x,y)\geq\bD(f(x),f(y))$. As a consequence of this characterization, when defining $\bR_+$-functors, we will
usually only define a function on objects and show it satisfies the required
inequality for hom-objects. 

\begin{remark}\label{rmk:Rfunctorsobjects}
	For categories $C$ and $D$ enriched over a monoidal category $\scr{V}$, an enriched functor $F:C\to D$ consists of a map $F:\on\Ob(C)\to \Ob(D)$ on objects, and morphisms of hom-objects in $\scr{V}$
	\[
	\begin{tikzcd}
		F_{a,b}:&[-3em] \scr{V}\on{-Hom}_C(a,b) \arrow[r] & \scr{V}\on{-Hom}_D(F(a),F(b))
	\end{tikzcd}
	\]
	for every pair of objects $a,b\in \on\Ob(C)$, satisfying compatibility conditions (see \cite[pg. 9]{kelly} for details). As such, two such enriched functors are equal precisely when their underlying maps on objects are equal, and for every pair of objects $a,b\in \Ob(C)$ the morphisms on hom-objects are equal. Since we are enriching in the \emph{poset} $\RR_+$, the latter condition is vacuous: any two morphisms from $a$ to $b$ in $\RR_+$ are equal. As such, to see that two $\RR_+$-functors are equal, it suffices to show that their underlying maps on objects are equal. 
\end{remark}

Lawvere's insight was that the definition of $\RR_+$-enriched categories can be
reformulated to be viewed as a weakening of the definition of metric spaces.

\begin{definition}
	A \textit{Lawvere metric space} $(X,d)$ consists of a set $X$, together with a function
	\[d:X\times X\to [0,\infty]\]
	satisfying the following conditions:\begin{itemize}
		\item For all $x\in X$, $d(x,x)=0$.
		\item (The triangle inequality) For all $x,y,z\in X$,
		\[d(x,y)+d(y,z)\geq d(x,z).\]
	\end{itemize}
\end{definition}

This is a weaker definition than that of a metric space: we allow for infinite distances, we do not require that $d(x,y)=0$ implies $x=y$, and do not require that $d(x,y)=d(y,x)$. 

\begin{remark}
	A $\bR_+$-enriched category is precisely the same thing as a Lawvere metric space. Similarly, an $\bR_+$ functor is precisely the data of a \textit{short map} $f:(X_1,d_1)\to(X_2,d_2)$, that is, a function $f:X_1\to X_2$ such that for all $x,y\in X_1$,
	\[d_1(x,y)\geq d_2(f(x),f(y)).\]
\end{remark}

Now, we introduce some terminology and notation concerning $\bR_+$-categories that will be useful later on.

\begin{definition}
    Given an $\bR_+$-category $\bC$ and objects $x$, $y$ in $\bC$, we say that $x$ and $y$ are \textit{isomorphic} (or \textit{indiscernible}) and write $x\cong y$ if 
	\[\bC(x,y)=0=\bC(y,x).\]
\end{definition}

\begin{definition}
	We say an $\bR_+$-category $\bC$ is \emph{gaunt} or \emph{satisfies the identity of indiscernibles} if, for every pair of objects $x$ and $y$  in $\bC$,
	\[x\cong y\iff x=y.\]
	
\end{definition}

\begin{definition}
    We say $\bC$ is \textit{symmetric} if
	\[\bC(x,y)=\bC(y,x)\]
	for all pairs of objects $x$ and $y$ in $\bC$.
\end{definition}

\begin{definition}
	Given an $\bR_+$-category $\bC$, we can form an $\bR_+$-category $\bC^\op$ with the same objects, and hom-objects given by 
	\[\bC^\op(x,y):=\bC(y,x).\]
	Note that $\bC=\bC^\op$ if and only if $\bC$ is symmetric.
\end{definition}

\begin{definition}\label{def:collected_R+defs}
	Let $f:\bC\to\bD$ be an $\bR_+$-functor. 
	\begin{enumerate}[label=(\arabic*)]
    \item We say $f$ is \textit{fully faithful} if
    $f_{x,y}:\mathbb{C}(x,y)\to\mathbb{D}(f(x),f(y))$ is an isomorphism in
    $\bR_+$ for all $x,y\in\bC$, i.e., if $\bD(f(x),f(y))=\bC(x,y)$ for all
    objects $x$ and $y$ in $\bC$.
		\item We say $f$ is \textit{essentially surjective} if for all objects $d$ in $\bD$, there exists some $c$ in $\bC$ such that $f(c)$ is isomorphic to $d$.
		\item We say that $f$ is an \textit{equivalence} of $\bR_+$-categories if it is both fully faithful and essentially surjective.
	\end{enumerate}
\end{definition}

Note that in the language of Lawvere metric spaces, a fully faithful $\bR_+$-functor is an \textit{isometry}, though we will use this term sparingly. Furthermore, if $f$ is an essentially surjective $\bR_+$-functor whose codomain is gaunt, then $f$ must be strictly surjective.

\begin{definition}
	We write $\bR_+\text-\Cat$ to denote the category of (small) $\bR_+$-categories, whose morphisms are $\bR_+$-functors. We further define $\bR_+\text-\Cat^\sym$ and $\bR_+\text-\Cat^\mathrm{gaunt}$ to be the full subcategories of $\bR_+\text-\Cat$ on the symmetric and gaunt $\bR_+$-categories, respectively.
\end{definition}

\begin{remark}\label{R+Cat_cocomplete}
  It follows from the general theory of enriched categories, including
  \cite[pg.\ 5]{R-CatComplete}, (or from a not-too-difficult direct
  construction) that $\RR_+\text-\Cat$ has all small limits and colimits. It is
  not hard to show that $\RR_+\text-\Cat^{\on{sym}}$ is closed under limits and
  colimits, and so is itself complete and cocomplete. Alternatively, in
  \cite[\S1]{Jardine}, Jardine explicitly constructs colimits in
  $\RR_+\text-\Cat^{\on{sym}}$, although there he calls symmetric
  $\RR_+$-categories \emph{ep-metric spaces}. 
\end{remark} 

One very important example of an $\bR_+$-category is the following.

\begin{definition}
	We can define a $\bR_+$-category whose objects are $[0,\infty]$ and the hom-object from $a$ to $b$ is given by $[a,b]:=\max(b-a,0)$.\footnote{Recall our convention for subtraction in $\bR_+$: $\infty-\infty=0$, $a-\infty=-\infty$, and $\infty-a=\infty$ for all $a<\infty$ in $\bR_+$.} By abuse of notation, we also use $\bR_+$ to denote this $\bR_+$-category.
\end{definition}

From this point onwards, we write $\bR_+$ to refer the $\bR_+$-enriched
category as defined above, not the poset $1$-category, unless stated
otherwise. Before discussing Cauchy completion, we give an $\bR_+$-enriched
definition of functor categories and the Yoneda embedding.

\begin{definition}
	Let $\bC$ and $\bD$ be two $\bR_+$-enriched categories. Define an $\RR_+$-category denoted $\bR_+\text-\Fun(\bC,\bD)$, or just $\bD^\bC$, whose objects are the class of $\bR_+$-functors from $\bC$ to $\bD$, and hom-objects are given by:
	\[\bD^\bC(f,g):=\sup_{c\in\bC}\bD(f(c),g(c)).\]
  It is straightforward to verify that the construction above indeed defines a
  valid $\bR_+$-category. Given an $\bR_+$-category $\bC$, we denote the
  category $\bR_+\text-\Fun(\bC^\op,\bR_+)$ by ${(\mathbb{R}_+)}_\bC$, and we
  refer to objects of this category as \emph{presheaves} on $\bC$.
\end{definition}

\begin{definition}\label{def:R-yoneda-embedding}
	Let $\bC$ be an $\bR_+$-category. Then the \emph{$\bR_+$-enriched Yoneda embedding} on $\bC$ is the assignment
	\[\begin{tikzcd}[row sep=0em]
		\cY_\bC: &[-3em] \bC\arrow[r] & (\mathbb{R}_+)_\bC \\
		& c\arrow[r, mapsto] & \bC(-, c),
	\end{tikzcd}\]
	where $\bC(-,c):\bC^\op\to\bR_+$ is the presheaf sending $c'\mapsto\bC(c',c)$.
\end{definition}

\begin{proposition}
	The $\bR_+$-enriched Yoneda embedding is a fully faithful $\bR_+$-functor.
\end{proposition}
\begin{proof}
	This is \autoref{prop:yoneda_fully_faithful} in the appendix.
\end{proof}

So far, we have discussed a bit of the general theory of $\bR_+$-enriched categories and $\bR_+$-enriched functors. In particular, we have seen that $\bR_+$-enriched categories are best thought of as generalizations of metric spaces. It remains to discuss the Cauchy completion of $\bR_+$-categories as constructed by Borceux and Dejean in \cite{BorceuxDejean}.

In \cite[Definition 1]{BorceuxDejean}, Borceux and Dejean more generally define
the $\mathcal{V}$-Cauchy completion of categories enriched in an arbitrary
enriching category $\cV$ (rather than just the case $\cV=\bR_+$). In the rather
technical definition (which we will not expand on here), they define the Cauchy
completion of a $\cV$-category $C$ to be the full $\cV$-subcategory of
the category $\cV\text-\mathrm{Prof}(I,C)$ (the category of
$\cV$-\textit{profunctors} from the monoidal unit to $C$) on the
$\cV$-profunctors that have right adjoints. In the case $\cV=\Set$, Borceux and
Dejean prove that this definition coincides exactly with our
\autoref{def:cauchy_completion_category} for the Cauchy completion of a
category (see \cite[Proposition
4]{BorceuxDejean}). In \autoref{alt_char:cauchy_completion}, we also prove a
nicer characterization of their definition in the $\bR_+$-enriched setting
(when $\cV=\bR_+$). 

In the beginning of the proof of \cite[Example 3]{BorceuxDejean}, Borceux and
Dejean further prove that given an $\bR_+$-category $\bC$, if the Lawvere
metric space associated with $\bC$ is an actual metric space, then the Cauchy
completion $\ol\bC$ of $\bC$ may be given as the full $\bR_+$-subcategory of
${(\bR_+)}_\bC$ on those presheaves that have \textit{duals}, in the following
sense.

\begin{definition}\label{def:dualpresheaf}
	Given a symmetric ${\bR_+}$-category $\bC$ and presheaves $f$ and $g$ in ${(\bR_+)}_\bC$, we say that $f$ and $g$ are \textit{dual} if
	\[0=\inf_{z\in\bC}(f(z)+g(z))\]
	and 
	\[f(c)+g(y)\geq\bC(x,y)\]
	for all $x,y\in\bC$.
\end{definition}

In fact, the same argument they give extends \textit{mutatis mutandis} to not
just $\bR_+$-categories that are metric spaces, but to all symmetric
$\bR_+$-categories. Thus, for symmetric $\bR_+$-categories, we can state
Borceux and Dejean's definition of the Cauchy completion as follows.

\begin{definition}\label{def:cauchy_completion}
  Given a symmetric $\bR_+$-category $\bC$, we define the
  \textit{Cauchy completion} of $\bC$ to be the full $\bR_+$-subcategory
  $\ol\bC$ of ${(\bR_+)}_\bC$ on those presheaves that have duals, in the sense
  of \autoref{def:dualpresheaf}.
\end{definition}

This definition is rather opaque, nor is it clear if $\ol\bC$ is a
``completion'' of $\bC$ in any sense. We will show that in the case the
underlying Lawvere metric space of $\bC$ is an actual metric space, the Cauchy
completion $\ol\bC$ of $\bC$ is isomorphic to the standard notion of the Cauchy
completion of $\bC$, thus justifying why this concept deserves to be called
``Cauchy completion''. To that end, the rest of this section will develop some
of the theory of analysis in $\bR_+$-categories/Lawvere metric spaces. Nearly
all of the proofs will be relegated to the appendix, but we include the
statements of those results that are relevant to our exposition. 

To start, we define sequences and subsequences in $\bR_+$-categories.

\begin{definition}
	Given an $\bR_+$-category $\bC$, a \emph{sequence} in $\bC$ is a countable collection of objects $x_1,x_2,x_3,\ldots$ in $\bC$ indexed by $\bN=\{1,2,3,\ldots\}$. We will denote sequences by $(x_n)_{n\in\bN}$, $(x_n)_{n=1}^\infty$, or just $(x_n)$. Occasionally, we may write $\overline x$ for the sequence $(x_n)$.

	Given a sequence $(x_n)_{n\in\bN}$ in an $\bR_+$-category $\bC$ and a sequence $(n_k)_{k\in\bN}$ of positive numbers such that $n_1<n_2<\cdots$, the sequence $(x_{n_k})_{k\in\bN}$ is called a \textit{subsequence} of $(x_n)$.
\end{definition}

Now, we may define a notion of Cauchy sequences and convergence.

\begin{definition}
	Given an $\bR_+$-category $\bC$ and a sequence $(x_n)$ in $\bC$, we say $(x_n)$ is a \emph{Cauchy sequence} if for all real numbers $\vare>0$ there exists some $N\in\bN$ such that
	\[n,m\geq N\ \implies\ \max(\bC(x_n,x_m),\bC(x_m,x_n))<\vare.\]
\end{definition}

\begin{definition}
	Given an $\bR_+$-category $\bC$, a sequence $(x_n)$ in $\bC$, and some object $c\in\bC$, we say that $c$ is a \textit{limit} of $(x_n)$ if for every $\vare>0$ there exists some $N\in\bN$ such that
	\[n\geq N\ \implies\ \max(\bC(x_n,c),\bC(c,x_n))<\vare.\]
	If $(x_n)$ has a limit, then we say that the sequence $(x_n)$ \textit{converges}. Furthermore, if a sequence $(x_n)$ in an $\bR_+$-category has a \textit{unique} limit $c$, we write
	\[\lim_{n\to\infty}x_n=c.\]
\end{definition}

Note that by unraveling definitions, the above definition may be expressed in
terms of limits in $\mathbb{R}_+$: given a sequence $(x_n)$ in a
$\mathbb{R}_+$-category $\mathbb{C}$, an element $c$ is a limit of $(x_n)$ if
and only if $0$ is a limit of both of the sequences
$(\mathbb{C}(x_n,c))_{n\in\mathbb{N}}$ and
$(\mathbb{C}(c,x_n))_{n\in\mathbb{N}}$.

Given an object $x$ in an $\bR_+$-category $\bC$, we will write $\wh x$ to
denote the constant sequence on $x$, which is clearly a Cauchy sequence with
a limit $x$. In a gaunt $\bR_+$-category, it can be seen that limits of sequences
are necessarily unique if they exist (\autoref{prop:isomorphic_limits}). In
particular, limits in $\bR_+$ are unique.

\begin{remark}
If we further adopt the convention that $|\infty|=|-\infty|=\infty$, then given $x,y\in\bR_+$ we have 
\begin{align*}
	\max({\bR_+}(x,y),{\bR_+}(y,x))&=\max(\max(y-x,0),\max(x-y,0))\\
	&=|x-y|=|y-x|.
\end{align*}
\end{remark}

In this way, for sequences of (non-infinite) real numbers, the usual notion of a limit corresponds with our definition of the limit given above. For this reason, when considering the convergence of sequences in $\bR_+$, we will replace the unwieldy $\max(\bR_+(x,y),\bR_+(y,x))$ in the preceding definitions with $|x-y|$. 

\begin{definition}\label{def:equivalence_of_Cauchy_seqs}
	Given an ${\bR_+}$-category $\bC$, we say that two sequences $(x_n)$ and $(y_n)$ are \textit{equivalent}, and write $(x_n)\sim(y_n)$, if
	\[\lim_{n\to\infty}\max(\bC(x_n,y_n),\bC(y_n,x_n))=0\]
	(where here the limit is taken in $\bR_+$). Unravelling definitions, $(x_n)$ and $(y_n)$ are equivalent if and only if, for all $\vare>0$, there exists some $N\in\bN$ such that
	\[n\geq N\implies\max(\bC(x_n,y_n),\bC(y_n,x_n))<\vare.\]
\end{definition}

In the case that $\bC$ is a metric space, the above definitions correspond
exactly to their usual meanings in analysis. Despite similarities between the
convergence of sequences in $\bR_+$-categories as opposed to metric spaces,
there are still distinctions. Namely, by our conventions for addition and
subtraction, a sequence in $\bR_+$ that grows arbitrarily does \textit{not}
converge to infinity. Rather, such a sequence does not converge at
all. Instead, the only sequences in $\bR_+$ that converge to infinity under our
definition are those that are eventually constant on $\infty$. Furthermore,
limits of sequences in $\mathbb{R}_+$-categories are, in general, not
unique. Indeed, consider the $\bR_+$ category $\bI=\{a_1,a_2\}$ where
$\bI(a_1,a_2)=0=\bI(a_2,a_1)$. It can be seen that every sequence in $\bI$
converges to both $a_1$ and $a_2$. These distinctions are important and
somewhat counterintuitive; we encourage the reader to keep them in mind.

Now, we continue our characterization of the Cauchy completion of symmetric
$\bR_+$-categories. For the remainder of this subsection, with the exception of
$\bR_+$, we will only consider \textit{symmetric} $\bR_+$-categories (although
we will still state this assumption when we use it). First, we define an
important presheaf associated to each Cauchy sequence in a symmetric
$\bR_+$-category $\bC$.

\begin{definition}\label{def:presheaf_associated_to_cauchy_sequence}
	Let $\bC$ be a symmetric ${\bR_+}$-category, and let $(x_n)$ be a Cauchy sequence in $\bC$. We fix notation and define a presheaf
	\[\ell_{(x_n)}:\bC^\op\to{\bR_+}\]
	by
	\[\ell_{(x_n)}(z):=\lim_{n\to\infty}\bC(z,x_n).\]
\end{definition}

The existence of the above limit is guaranteed by
\autoref{lemma:lim_of_cauchy_seqs_exists} from the appendix. The assignment
$\ell_{(x_n)}$ is proven to be an $\bR_+$-functor in
\autoref{prop:ell_x_is_short_map}. With these notions in hand, we may now state
the following proposition, which provides an alternative characterization of
\autoref{def:cauchy_completion}.

\begin{proposition}\label{prop:equivalent_Cauchy_characterization_fake}
	Let $\bC$ be a symmetric $\RR_+$-category and $f:\CC^\op\to \RR_+$ an $\RR_+$-enriched presheaf. The following are equivalent. 
	\begin{enumerate}
		\item The presheaf $f$ is a limit (in $(\RR_+)_\CC$) of a Cauchy sequence of objects in the image of the Yoneda embedding. 
		\item The presheaf $f$ has a dual, i.e., $f$ belongs to the Cauchy completion $\ol\bC$ of $\bC$ (\autoref{def:cauchy_completion}).
		\item There is a Cauchy sequence $(x_n)$ in $\bC$ such that $f=\ell_{(x_n)}$.
	\end{enumerate}
\end{proposition}
\begin{proof}
	This is \autoref{prop:equivalent_Cauchy_characterization} in the appendix.
\end{proof}

For the remainder of the article, we will nearly always use the
characterization of the Cauchy completion $\ol\bC$ as the $\bR_+$-category
whose objects are the presheaves $\ell_{(x_n)}$ for Cauchy sequences $(x_n)$ in
$\bC$, and whose hom-objects are given by
\[
  \ol\bC(\ell_{(x_n)},\ell_{(y_n)})
  =\lim_{n\to\infty}\bC(x_n,y_n).
\]

Recall that in the standard theory of metric spaces, given a metric space $(X,d)$, its \textit{Cauchy completion} is defined to be the metric space $(\ol X,\ol d)$, where $\ol X$ is the set of equivalence classes $[(x_n)]$ of Cauchy sequences in $(X,d)$, and the metric is defined by 
\[\ol d([(x_n)],[(y_n)]):=\lim_{n\to\infty} d(x_n,y_n).\]

\begin{proposition}\label{alt_char:cauchy_completion}
  If $\bC$ is a symmetric $\bR_+$-category whose underlying Lawvere metric
  space is an actual metric space, then the Cauchy completion $\ol{\bC}$
  defined in \autoref{def:cauchy_completion} corresponds precisely to the usual
  notion of the Cauchy completion of $\bC$ as a metric space.
\end{proposition}
\begin{proof}
In \autoref{prop:equiv-cauchy-same-presheaf}, we show that given two Cauchy sequences $(x_n)$ and $(y_n)$ in $\bC$, they are equivalent (in the sense of \autoref{def:equivalence_of_Cauchy_seqs}) if and only if their corresponding presheaves $\ell_{(x_n)}$ and $\ell_{(y_n)}$ are equal. Hence, by \autoref{prop:equivalent_Cauchy_characterization_fake}, the objects of $\ol\bC$ may be canonically identified with the equivalence classes of Cauchy sequences in $\bC$. The necessary characterization of the hom-objects is proven in \autoref{prop:Lmetric-on-cauchy-completion}.
\end{proof}

Note that given a symmetric $\bR_+$-category $\bC$ and an object $x$ in $\bC$,
it is straightforward to see from the definitions that the presheaf $\ell_{\wh x}$
associated to the constant Cauchy sequence $\wh x$ on $x$ is precisely the
presheaf $\cY_\bC(x)=\bC(-,c)$. As a consequence, by
\autoref{prop:equivalent_Cauchy_characterization_fake}, we have the following
remark, that introduces some notation for the functor
$\iota_{\mathbb{C}}:\mathbb{C}\to \overline{\mathbb{C}}$, which we will use
frequently in the remainder of the document.

\begin{remark}\label{remark:iota_C_Cauchy_completion}
  Given a symmetric $\bR_+$-category $\bC$, the $\bR_+$-enriched Yoneda
  embedding (\autoref{def:R-yoneda-embedding}) determines a canonical fully
  faithful inclusion $\bR_+$-functor $\iota_\bC:\bC\into\ol\bC$ sending an
  object $x$ to the presheaf $\ell_{\wh x}=\mathbb{C}(-,x)$.
\end{remark}

We say a symmetric $\bR_+$-category $\bC$ is \textit{Cauchy complete} if the
inclusion $\iota_\bC:\bC\into\ol\bC$ is an $\bR_+$-equivalence of categories,
or equivalently, if every Cauchy sequence in $\bC$ has a limit.  In
\autoref{prop:induced_R-functor_between_completions}, we show that given an
$\bR_+$-functor $f:\bC\to\bD$ between symmetric $\bR_+$-categories, there is a unique induced functor $\ol f:\ol\bC\to\ol\bD$ satisfying
$\overline{f}\circ\iota_{\mathbb{C}}=\iota_{\mathbb{D}}\circ f$. Furthermore,
in \autoref{prop:R_pastoral_equiv} we prove that $\ol f$ is an
$\bR_+$-equivalence of categories if and only if $f$ is fully faithful and
dense, in the sense of \autoref{def:dense_R+_functor}.


\section{The Karoubian model structure on \texorpdfstring{$\Cat$}{Cat}}\label{section:Karoubian}
Now we turn to homotopy theory. Our first result is a proof of the existence of
a model structure on $\Cat$ that we call the \emph{Karoubian model
structure}, following \cite[\S2]{JcatLabModels}. This model structure
characterizes the Cauchy completion of categories in that the
fibrant-cofibrant objects are the Cauchy
complete categories, as discussed in \Cref{section:preliminaries}. In
\cite[\S1]{MoritaHmtpy}, this model structure is called the \textit{Morita
model structure} on $\Cat$, and is constructed as a Bousfield localization of
the canonical model structure on $\Cat$. This section offers an
alternative approach to the construction of this model structure. For the
definition of a model structure, we refer the reader to \cite[\S2.1]{GoerssJardine}. We will use the numbering of the axioms given
there; it is also the numbering Rezk uses in the proof of the canonical model
structure on $\Cat$ in \cite{RezkCat}. To start, we define the weak
equivalences, fibrations, and cofibrations in the Karoubian model structure.

\begin{definition}
	We say that a functor $F$ is a \textit{idfibration} if it has the right lifting property with respect to the canonical inclusion $\Sigma:\Idem\to\Split$. Let $\IdFib$ denote the class of all idfibrations between categories.
\end{definition}

\begin{theorem}\label{theorem:Karoubian_model_structrue}
	There is a model structure on the category $\Cat$ of small categories such that  \begin{itemize}
		\item[($\cW$)] The weak equivalences are the pastoral equivalences.
		\item[($\Cof$)] The cofibrations are the functors that are injective on objects. 
		\item[($\Fib$)] The class of fibrations is given by $\IdFib$. 
	\end{itemize}
	Furthermore, the fibrant-cofibrant objects in this model structure are precisely the small Cauchy complete categories. We call this model structure the \emph{Karoubian model structure} on $\Cat$. 
\end{theorem}

The characterization of the fibrant-cofibrant objects is entirely
straightforward. Now we embark on a proof that the Karoubian model structure is
a model structure, by verifying the five required axioms. Throughout
this proof, we primarily use the characterization of pastoral equivalences as
fully faithful functors that are surjective up to retracts
(\autoref{proposition:surj_retract_equivalence}). Furthermore, our proof relies
heavily on the proof of the canonical model structure given in
\cite{RezkCat}. Notice that we have defined the cofibrations in the
Karoubian model structure to be the functors that are injective on objects;
this is the same as the definition of the cofibrations in the canonical model
structure. As a consequence, we have the following preliminary lemma:

\begin{lemma}\label{lemma:triv_fibs_in_Karoubian_and_canonical_coincide}
  The canonical model structure and the Karoubian model structure on $\Cat$
  have the same trivial fibrations.
\end{lemma}
\begin{proof}
  In the proof of Axiom M4 in \cite{RezkCat}, Rezk proves that the trivial
  fibrations in the canonical model structure on $\Cat$ are precisely the
  strictly surjective fully faithful functors, which are clearly trivial
  fibrations in the Karoubian model structure on $\Cat$. On the other hand,
  suppose $\pi:\cC\to\cD$ is a trivial fibration in the Karoubian model
  structure. We know $\pi$ is fully faithful. Since $\pi$ is surjective up to
  retracts and has the right lifting property against $\Sigma$, it can be seen
  that $\pi$ is in fact a strictly surjective functor.
\end{proof}

Lastly, we introduce the following notation:

\begin{notation}
  Given a functor $F:\cC\to\cD$ and objects $x$ and $y$ in $\cC$, we write
  $F_{x,y}$ to denote the induced function
  \[
    \Hom_\cC(x,y)\to\Hom_\cD(F(x),F(y)).
  \]
\end{notation}

Now we may prove \autoref{theorem:Karoubian_model_structrue}.

\subsection*{Axiom CM1} It is well known that $\Cat$ has all small limits and colimits.

\subsection*{Axiom CM2} By Axiom CM2 for the canonical model structure on
$\Cat$ (in which the weak equivalences are the equivalences of categories, see
\cite{RezkCat}), it is clear that if two of $F$, $G$, and $F\circ G$ are
pastoral equivalences then so is the third, by passing to the post-composition
functors and applying \autoref{proposition:surj_retract_equivalence}.

\subsection*{Axiom CM3} The cofibrations in the Karoubian model structure
coincide with those in the canonical model structure on $\Cat$, so it follows
$\Cof$ is closed under retracts. As with Axiom CM2, it follows from
the canonical model structure on $\Cat$ that $\cW$ is closed under retracts, by
passing a retract diagram to its post-composition functors. Since the
fibrations are characterized by a lifting property, they are also closed under
retracts (\cite[Lemma 11.1.4]{RiehlCatHmptyThry}).

\subsection*{Axiom CM4}
Suppose we have a lifting problem in $\Cat$ of the following form
\[
\begin{tikzcd}
	{\cA} & {\cC} \\
	{\cB} & {\cD}
	\arrow["{\iota}"', from=1-1, to=2-1]
	\arrow["{\pi}", from=1-2, to=2-2]
	\arrow["{F}", from=1-1, to=1-2]
	\arrow["{G}", from=2-1, to=2-2]
\end{tikzcd}  
\]
where $\iota$ is a cofibration and $\pi$ is an idfibration. We want to show that if either $\iota$ or $\pi$ is a pastoral equivalence, then the lifting problem has a solution. 

\textbf{Case 1.} First suppose $\iota$ is a pastoral equivalence, so for each object $b$ in $\cB$ we can choose an object $a_b$ in $\cA$ and a retraction
\[\begin{tikzcd}
	b\arrow[r,"i_b"]&\iota(a_b)\arrow[r,"r_b"]& b.
\end{tikzcd}\]
In the case $b$ is in the image of $\iota$ since $\iota$ is injective on objects, we can choose $a_b$ to be the unique object in $\cA$ such that $\iota(a_b)=b$, and we set $r_b=i_b=\id_b$. Since $\iota$ is fully faithful, there is a unique arrow $h_b:a_b\to a_b$ such that $\iota(h_b)=i_b\circ r_b$. Moreover, $h_b$ is an idempotent in $\cA$ by faithfulness of $\iota$. Note that when $b$ is in the image of $\iota$, by our earlier choices we have $h_b=\id_{a_b}$. Now consider the lifting problem
\[
\begin{tikzcd}
	{\Idem} & {\cC} \\
	{\Split} & {\cD,}
	\arrow["{\Sigma}"', from=1-1, to=2-1]
	\arrow["{\pi}", from=1-2, to=2-2]
	\arrow["{\nu_b}", from=1-1, to=1-2]
	\arrow["{\psi_b}", from=2-1, to=2-2]
\end{tikzcd}  
\]
defined by setting $\nu_b(e)=F(h_b)$, $\psi_b(p)=G(r_b)$, and $\psi_b(q)=G(i_b)$. The diagram commutes since
\[\pi(\nu_b(e))=\pi(F(h_b))=G(\iota(h_b))=G(i_b\circ r_b)=\psi_b(q\circ p)=\psi_b(\Sigma(e)).\] 
Since $\pi$ is an idfibration, we may choose a solution $L_b$ to this lifting problem. When $b$ is in the image of $\iota$, our earlier choices specified  $L_b$ to be the constant functor on $F(a_b)$. Now, we define the lift $L:\cB\to\cC$ to the original lifting problem. Given an object $b$ in $\cB$ we define $L(b):=L_b(1)$. Given an arrow $f:b\to d$, since $\iota$ is fully faithful, there exists a unique arrow $\wt f:a_b\to a_d$ such that
\[\iota(\wt f)=i_d\circ f\circ r_b.\]
Then define
\[L(f):=L_d(p)\circ F(\wt f)\circ L_b(q).\]
It is straightforward to check that since $L_b(0)=\nu_b(0)$ and $L_d(0)=\nu_d(0)$ that this composition is well-defined.
Furthermore, we claim $L$ is a functor. Given $f:b\to d$ and $g:d\to e$ in $\cB$, we have
\[L(g)\circ L(f)=L_e(p)\circ F(\wt g)\circ L_d(q)\circ L_d(p)\circ F(\wt f)\circ L_b(q).\]
Then using the fact that
\[L_d(q\circ p)=\nu_d(e)=F(h_d),\]
we get that
\[L(g)\circ L(f)=L_e(p)\circ F(\wt g\circ h_d\circ\wt f)\circ L_b(q).\]
Thus, by how we defined $L$, in order to show $L(g)\circ L(f)=L(g\circ f)$, it suffices to show that $\wt g\circ h_d\circ\wt f=\wt{g\circ f}$. By uniqueness, it further suffices to check that
\[\iota(\wt g\circ h_d\circ\wt f)=i_e\circ g\circ f\circ r_b.\]
Yet this is clear: Since $\iota(h_d)=i_d\circ r_d$ and $r_d\circ i_d=\id_d$, we have
\[\iota(\wt g\circ h_d\circ\wt f)=(i_e\circ g\circ r_d)\circ(i_d\circ r_d)\circ(i_d\circ f\circ r_b)=i_e\circ g\circ f\circ r_b,\]
as desired. Now, it remains to show the functor $L:\cB\to\cC$ is a solution to the original lifting problem, i.e., that $L\circ\iota=F$ and $\pi\circ L=G$. To see the former, note that by how we made our choices when constructing $L$, given an object $a$ in $\cA$, we have
\[L(\iota(a)):=L_{\iota(a)}(1)=F(a).\]
Similarly, given an arrow $f:a\to b$ in $\cA$, we know 
\[\iota(f)=\id_{\iota(b)}\circ\iota(f)\circ\id_{\iota(a)}=i_{\iota(b)}\circ\iota(f)\circ r_{\iota(a)},\] 
meaning $\wt{\iota(f)}=f$. Thus
\[L(\iota(f))=L_{\iota(b)}(p)\circ F(\wt{\iota(f)})\circ L_{\iota(a)}(q)=\id_{F(b)}\circ F(f)\circ\id_{F(a)}=F(f),\]
as desired. To see $\pi\circ L=G$, first note that given an object $b$ in $\cB$ we have
\[\pi(L(b))=\pi(L_b(1))=\psi_b(1)=\dom\psi_b(q)=\dom G(i_b)=G(b),\]
and given an arrow $f:b\to d$ in $\cB$, we have
\[\pi(L(f))=\pi(L_d(p)\circ F(\wt f)\circ L_b(q)).\]
Then using the fact that $\pi\circ L_b=\psi_b$, $\pi\circ F=G\circ\iota$, and $\iota(\wt f)=i_d\circ f\circ r_b$, and $r_b\circ i_b=\id_{b}$, we get
\[\pi(L(f))=G(r_d)\circ G(i_d\circ f\circ r_b)\circ G(i_b)=G(f).\]
Thus, we've constructed a functor $L:\cB\to\cC$ such that $L\circ\iota=F$ and $\pi\circ L=G$, as desired.

\textbf{Case 2:} Suppose $\iota\in\Cof$ and $\pi\in\cW\cap\IdFib$. In this case, by \autoref{lemma:triv_fibs_in_Karoubian_and_canonical_coincide}, $\pi$ is a trivial fibration in the canonical model structure on $\Cat$, and $\iota$ is a cofibration in the canonical model structure, so there is a lift.

\subsection*{Axiom CM5}
We now prove the two requisite factorization properties of morphisms in $\Cat$. Given a functor $F:\scr{C}\to \scr{D}$, by Rezk's proof of Axiom CM5 for the canonical model structure on $\Cat$ \cite{RezkCat}, every morphism factors as a cofibration followed by a trivial fibration in the canonical model structure on $\Cat$. The factorization into a cofibration and trivial fibration then follows by \autoref{lemma:triv_fibs_in_Karoubian_and_canonical_coincide}. 

On the other hand, since $\Cat$ is locally presentable (see \cite{Catlocpres} for a proof in greater generality), we can apply the small object argument \cite[Theorem 2.1.14]{Hovey_1999} to the set of morphisms $\{\Idem\to \Split\}$, yielding a factorization of $F:\scr{C}\to \scr{D}$ into 
\[
\begin{tikzcd}
	\scr{C}\arrow[r,"\iota"] &\scr{L}\arrow[r,"\pi"] & \scr{D}
\end{tikzcd}
\]
where $\pi$ is an idfibration, and $\iota$ is a transfinite composite of pushouts of coproducts copies of $\Idem\to \Split$. We thus need only see that the latter are pastoral equivalences and injective on objects. 

Functors that are pastoral equivalences and injective on objects are closed under coproducts. To see they are closed under transfinite composition, let
\[
\begin{tikzcd}
	\scr{C}_0\arrow[ddrrrr,"G_0"']\arrow[r,"F_1"] & \scr{C}_1 \arrow[r,"F_2"]\arrow[ddrrr] & \scr{C}_2 \arrow[r]\arrow[ddrr] & \cdots  & \\
	& & & &  \\
	& & & & \scr{E} 
\end{tikzcd}
\] 
be a colimit diagram in which the functors $F_i$ are pastoral equivalences and injective. By \autoref{proposition:surj_retract_equivalence} these are functors that are fully faithful, injective, and surjective up to retracts. We wish to show that $G_0$ is a fully faithful functor that is injective and surjective up to retracts. Since the functor that sends a category to its set of objects preserves colimits, it follows that $G_0$ is injective on objects. It then follows that for objects $a$ and $b$ of $\scr{E}$ with representatives $a_i,b_i\in \Ob(\scr{C}_i)$ for $i$ sufficiently large, $\scr{E}(a,b)$ is the colimit of $\scr{C}_i(a_i,b_i)$. Thus, we see that $G_0$ is fully faithful. Finally, let $a$ be an object of $\scr{E}$. Then choose $i\in \NN$ and $a_i\in \Ob(\scr{C}_i)$ such that $a_i$ is sent to $a$ in the colimit. Define $F_{1,i}$ as the composition of $F_1$ through $F_i$. Since finite composites of functors that are surjective up to retracts are themselves surjective up to retracts, there is a retract diagram 
\[
\begin{tikzcd}
	a_i \arrow[r,"q"] & F_{1,i}(c)\arrow[r,"p"] & a_i  
\end{tikzcd}
\]
in $\scr{C}_i$. The image of this diagram in $\scr{E}$ is a retract diagram displaying $a$ as a retract of $G_0(c)$. 

Finally, we show functors that are pastoral equivalences and injective on objects are closed under pushouts. Let 
\[
\begin{tikzcd}
	\scr{C}\arrow[d,"G"']\arrow[r,"F"] & \scr{D}\arrow[d,"H"]\\
	\scr{E}\arrow[r,"K"'] & \scr{P}
\end{tikzcd}
\]
be a pushout diagram in which $F$ is a pastoral equivalence and injective on objects. It follows from \cite[Proposition 5.2]{HomotopyNerve} that $K$ is fully faithful and injective on objects as well. Finally, if $a\in \Ob(\scr{P})$ it is either in the image of $H$ or $K$. Surjectivity up to retracts is clear in the latter case, since every object is a retract of itself. If $a$ is $H(d)$ for some $d\in \Ob(\scr{D})$, there is a retract diagram in $\scr{D}$ displaying $d$ as a retract of $F(c)$ for some $c\in \Ob(\scr{C})$. Applying $H$ to this diagram, we see that $a$ is a retract of $K(G(c))$, completing the proof.

\section{The Metric model structure on \texorpdfstring{$\bR_+\text-\Cat$}{R+-Cat}}\label{section:metric_model_structure}

We now turn our attention to the homotopy theory of $\bR_+$-categories. This section aims to show that on each of the categories $\RCat$ and $\RCat^{\on{sym}}$, there is a model structure \emph{uniquely determined} by the following properties:

\begin{itemize}
	\item The weak equivalences are the $\RR_+$-functors that are essentially surjective and fully faithful, i.e., the equivalences of $\RR_+$-categories. 
	\item Not every object is both fibrant and cofibrant.
\end{itemize}

In fact, these properties result in a model structure whose fibrant-cofibrant objects are the gaunt $\bR_+$-categories. Such a model structure is the natural first step to understanding metric spaces using homotopy-theoretic techniques since the homotopy theory it encodes is precisely the theory of extended metric spaces. This section is split into two subsections. In the first subsection, we will prove a model structure with these properties (if one exists) is unique. In the second subsection, we prove that such a model structure does exist. Before we begin, we define some notation that will be useful throughout the rest of the document.

\begin{notation}
	Given a class of morphisms $\cP$ in a category $\cC$, we write ${_\perp}\cP$ (resp.\ $\cP{_\perp}$) to denote the class of morphisms with the left (resp.\ right) lifting property against $\cP$.
\end{notation}

Note that if $\cP\sseq\cQ$ then ${_\perp}\cQ\sseq{_\perp}\cP$ and $\cQ{_\perp}\sseq\cP{_\perp}$. Furthermore, in the following proofs, we will use the fact that in a model structure with weak equivalences $\cW$, cofibrations $\Cof$, and fibrations $\Fib$, 
\[
  {_\perp}(\cW\cap\Fib)=\Cof,\quad
  \Cof{_\perp}=\cW\cap\Fib,\quad
  (\cW\cap\Cof){_\perp}=\Fib,
  \quad\text{and}\quad
  {_\perp}\Fib=\cW\cap\Cof
\]
(see \cite[Lemma 9.6]{GoerssJardine}). 

\begin{notation}\label{I_Delta_Gamma_defns}
  We fix the following common $\bR_+$-categories and $\bR_+$-functors, which we
  will use many times throughout the rest of the paper.
  \begin{itemize}
    \item Define $\bI$ to be the $\bR_+$-category with two elements $a_1$ and
    $a_2$ such that $\bI(a_1,a_2) = 0 = \bI(a_2, a_1)$ 
    \item Write $\ast$ for be the unique $\bR_+$-category with one object.
    \item Let $\Delta$ denote the unique $\mathbb{R}_+$-functor
    $\mathbb{I}\to\ast$.
    \item Let $\Gamma:\ast\to\bI$ denote the $\mathbb{R}_+$-functor sending
    $\ast\mapsto a_1$.
  \end{itemize}

  Note that the $\RR_+$-category $\mathbb{I}$ plays a role in our constructions completely analogous to the role played in \cite{RezkCat} by the walking isomorphism $I$. Both $\mathbb{I}$ and $I$ are examples of what are, in model category theory, called \emph{interval objects}, and are used in much the same way the topological interval $[0,1]$ is used in the homotopy theory of topological spaces.
\end{notation}

A straightforward examination yields the following characterization of those $\bR_+$-functors with the right lifting property against $\Delta$ and $\Gamma$:

\begin{remark}\label{remark:characterization_of_Delta_Gamma_lifting_properties}
    Let $f:\bC\to\bD$ be an $\bR_+$-functor, so given an object $c$ in $\bC$,
    $f$ maps the isomorphism class of objects $[c]$ into the isomorphism class
    $[f(c)]$. Then $f$ has the right lifting property against $\Delta$ (resp.\
    $\Gamma$) if and only if the induced assignment $[c]\to[f(c)]$ is injective
    (resp.\ surjective) for all objects $c$ in $\bC$. In particular,
    $f\in\{\Delta,\Gamma\}{_\perp}$ if and only if $f$ induces bijections of
    isomorphism classes $[c]\cong[f(c)]$.
\end{remark}

\begin{definition}\label{def:isofib}
    We define the \textit{isofibrations} in $\bR_+$-$\Cat$ to be the class $\{\Gamma\}{_\perp}$ of $\bR_+$-functors that have the right lifting property with respect to $\Gamma$.
\end{definition}

\begin{lemma}\label{lemma:delta_gamma_rlp_are_isomorph}
    If $\cW$ denotes the class of equivalences of $\bR_+$-categories and $\Iso$ denotes the class of isomorphisms of $\bR_+$-categories, then $\cW\cap\{\Delta,\Gamma\}{_\perp}=\Iso$.
\end{lemma}
\begin{proof}
    It is immediate that $\Iso\sseq\cW\cap\{\Delta,\Gamma\}_\perp$, as every isomorphism is a weak equivalence, and the isomorphisms clearly have the right lifting property with respect to every morphism. To show the other inclusion, let $f:\bC\to\bD$ be an $\bR_+$-functor belonging to $\cW\cap\{\Delta,\Gamma\}_\perp$. Since $f$ is fully faithful, it suffices to show $f$ induces a bijection between the objects of $\bC$ and the objects of $\bD$. Let $c_1$ and $c_2$ be objects in $\bC$ such that $f(c_1)=f(c_2)$. Since $f$ is fully faithful, $c_1$ is isomorphic to $c_2$, i.e., $\bC(c_1,c_2)=0$. Then since $f$ has the right lifting property with respect to $\Delta$, it follows that $c_1=c_2$, hence $f$ is injective. Let $d$ be an object $\DD$, then since $f$ is essentially surjective, there is an object $c$ in $\CC$ such that $f(c)\cong d$. Since $f$ has the right lifting property against $\Gamma$, we may lift this isomorphism to $\bC$, yielding an object $c_d\in \CC$ with $f(c_d)=d$. Thus $f$ is surjective, and therefore an isomorphism.
\end{proof}

\subsection{Uniqueness of the Metric model structure}

This subsection aims to prove the aforementioned claim of uniqueness:

\begin{theorem}\label{thm:metric_MS_unique}
	There is at most one model structure on $\bR_+\text-\Cat$ (resp.\ $\bR_+\text-\Cat^\sym$) satisfying the following conditions:\begin{enumerate}
		\item The weak equivalences are the $\bR_+$-functors that are essentially surjective and fully faithful, i.e., the equivalences of $\bR_+$-categories.
		\item Not every object is both fibrant and cofibrant.
	\end{enumerate}
\end{theorem}

We will prove \autoref{thm:metric_MS_unique} in the non-symmetric case, as
proving it in the symmetric case is entirely identical. Our proof roughly
follows the blogpost \cite{CatBlog}, which proves a similar uniqueness
result for the canonical model structure on $\Cat$. For the remainder of this
subsection, we will assume $\mathcal{M}$ is a model structure on $\RCat$
satisfying conditions (1) and (2) given above. We denote the classes of weak
equivalences, cofibrations, and fibrations in this model structure by $\cW$,
$\Cof$, and $\Fib$, respectively. Our work will culminate in
\autoref{lemma:triv_fib_equal_isos}, where we will show that the trivial
fibrations in $\cM$ must be precisely the strict isomorphisms of
$\bR_+$-categories. In other words, we will have uniquely determined the weak
equivalences and the trivial fibrations, from which uniqueness follows.

\begin{lemma}\label{lemma:0_to_b_cof}
	In the model structure $\mathcal{M}$,  $\emptyset\to\ast$ is a cofibration.
\end{lemma}
\begin{proof}
	By applying Axiom CM5 for a model category, we can factor $\emptyset\to\ast$ as $\emptyset\to\wh\ast\to\ast$, where $\wh\ast$ is cofibrant and the map $\wh\ast\to\ast$ is a trivial fibration. Note that $\wh\ast$ must be nonempty for $\wh\ast\to\ast$ to be essentially surjective, and in particular, a weak equivalence. Thus, there exists an $\bR_+$-functor $\ast\to\wh\ast$ with which we may construct a retract diagram
	\[\begin{tikzcd}
		\emptyset & \emptyset & \emptyset \\
		\ast & \widehat\ast & \ast
		\arrow[from=1-1, to=1-2]
		\arrow[from=1-2, to=1-3]
		\arrow[from=1-1, to=2-1]
		\arrow[from=2-1, to=2-2]
		\arrow[from=2-2, to=2-3]
		\arrow[from=1-2, to=2-2]
		\arrow[from=1-3, to=2-3]
	\end{tikzcd}\]
	The desired result follows by axiom CM3 for a model category.
\end{proof}

In what follows, write $\Inj$ and $\Surj$ to denote the classes of $\bR_+$-functors that are injective (resp.\ surjective) on objects.

\begin{lemma}\label{cor:inj_is_cof}
	$\Cof\spseq\Inj$.
\end{lemma} 
\begin{proof}
  The trivial fibrations must have the right lifting property with
  respect to all cofibrations, including $\emptyset\to\ast$ by
  \autoref{lemma:0_to_b_cof}. It is straightforward to check that the
  $\bR_+$-functors with the right lifting property against this map are
  precisely the $\bR_+$-functors that are surjective on objects, so it follows
  that the trivial fibrations in $\cM$ form a subclass of the class of
  $\bR_+$-functors that are surjective on objects, i.e.,
  $\cW\cap\Fib\sseq\Surj$. In particular,
  $\cW\cap\Fib\sseq\cW\cap\mathcal{S}\mathrm{urj}$, so that
  \[
    {_\perp}(\cW\cap\mathcal{S}\mathrm{urj})\sseq{_\perp}(\cW\cap\Fib)=\Cof.
  \]
  Hence, it suffices to show that the $\bR_+$-functors in $\Inj$ have the left
  lifting property against those in $\cW\cap\Surj$. Let 
	\[
	\begin{tikzcd}
		\mathbb{A} \arrow[r,"u"]\arrow[d,"i"'] & \mathbb{C}\arrow[d,"f"]\\
		\mathbb{B} \arrow[r,"v"'] & \mathbb{D}
	\end{tikzcd}
	\]
	be a commutative diagram in $\bR_+\text-\Cat$ such that $f$ is an equivalence that is surjective on objects, and $i$ is injective on objects. Since injective maps of sets have the left lifting property against surjective maps of sets, we can lift the underlying diagram of object sets to a morphism $\ell: \Ob(\bb{B})\to \Ob(\bb{C})$ making the diagram commute. We thus need only check that $\ell$ is an $\RR_+$-functor. However, given $a,b\in \Ob(\bb{B})$, we have that 
	\[
	\bb{B}(a,b) \geq \bb{D}(v(a),v(b))= \bb{D}(f(\ell(a)),f(\ell(b)))=\bb{C}(\ell(a),\ell(b))
	\]
	where the last equality follows because $\ell$ is a isometry, and the first inequality follows because  $v$ is an $\bR_+$-functor. Thus, $\ell$ is an $\bR_+$-functor, as desired. 
\end{proof}

In fact, we claim that $\Cof$ is strictly larger than $\Inj$:

\begin{lemma}\label{lemma:inj_cof_objects}
	$\Cof\neq\Inj$.
\end{lemma}
\begin{proof}
  Suppose for the sake of a contradiction that $\Cof=\Inj$. Then, by an
  argument very similar to that given in the proof of Axiom M4 in
  \cite{RezkCat}, the fibrations in $\cM$ are precisely the isofibrations
  (\autoref{def:isofib}). Now, let $\mathbb{C}$ be an arbitrary
  $\mathbb{R}_+$-category. First note that $\mathbb{C}\to\ast$ is an
  isofibration, so we would have that $\mathbb{C}$ is fibrant. Moreover, since
  $\emptyset\to\bC$ is vacuously injective on objects, so every object is
  cofibrant. This is a contradiction of condition (2).
\end{proof}

Thus, the class of cofibrations contains more than just those $\bR_+$-functors that are injective on objects.

\begin{lemma}\label{Delta_Gamma_are_trivial_cofibrations_uniqueness_proof_lemma}
	The $\bR_+$-functors $\Delta$ and $\Gamma$ are trivial cofibrations in $\cM$.
\end{lemma}
\begin{proof}
  It is clear to see $\Gamma$ is a weak equivalence and $\Delta$ is a weak
  equivalence. Moreover, $\Gamma$ is a cofibration as it is injective. To show
  $\Delta$ is a cofibration, we follow a proof roughly based on an argument
  given in \cite{CatBlog}. By \autoref{cor:inj_is_cof} and
  \autoref{lemma:inj_cof_objects}, there exists some cofibration $f:\bA\to\bB$
  and distinct objects $x$ and $y$ in $\bA$ such that $f(x)=f(y)$. Now let $g$
  denote the $\bR_+$-functor $\bA\to\bI$ that sends $x$ to $a_1$ and every
  other object to $a_2$. Then we may form the following pushout diagram in
  $\bR_+\text-\Cat$
	\[\begin{tikzcd}
		\bA & \bI \\
		\bB & {\bB\amalg_\bA\bI}
		\arrow["g", from=1-1, to=1-2]
		\arrow["{\wt f}", from=1-2, to=2-2]
		\arrow["f"', from=1-1, to=2-1]
		\arrow["{\wt g}", from=2-1, to=2-2]
		\arrow["\lrcorner"{anchor=center, pos=0.125, rotate=180}, draw=none, from=2-2, to=1-1]
	\end{tikzcd}\]
	Since cofibrations in a model category are closed under taking pushouts (\cite[Corollary 1.1.11]{Hovey_1999}), $\wt f$ is a cofibration. Furthermore, we have
	\[\wt f(a_1)=\wt f(g(x))=\wt g(f(x))=\wt g(f(y))=\wt f(g(y))=\wt f(a_2),\]
	so $\wt f$ factors through $\ast$ via $\Delta$, that yields the following retract diagram:
	\[\begin{tikzcd}
		\bI & \bI & \bI \\
		\ast & {\bB\amalg_\bA\bI} & \ast
		\arrow[Rightarrow, no head, from=1-1, to=1-2]
		\arrow[Rightarrow, no head, from=1-2, to=1-3]
		\arrow["\Delta", from=1-3, to=2-3]
		\arrow["\Delta"', from=1-1, to=2-1]
		\arrow["{\wt f(a_1)}"', from=2-1, to=2-2]
		\arrow[from=2-2, to=2-3]
		\arrow["{\wt f}", from=1-2, to=2-2]
	\end{tikzcd}\]
	Since the class of cofibrations in $\cM$ is closed under retracts, it follows that $\Delta$ is a cofibration, as desired.
\end{proof}

As a corollary of this result, we have the following.

\begin{corollary}\label{prop:fib_in_Delta_Gamma}
	$\Fib\subseteq \{\Delta,\Gamma\}_\perp$
\end{corollary}

Finally, we may explicitly describe the trivial fibrations in $\cM$.

\begin{lemma}\label{lemma:triv_fib_equal_isos} $\cW\cap\Fib=\Iso$.
\end{lemma}
\begin{proof}
 Note that the isomorphisms lift against every morphism,
  so we have $\Iso\sseq{(\Cof)}_\perp=\cW\cap\Fib$. By
  \autoref{prop:fib_in_Delta_Gamma} and
  \autoref{lemma:delta_gamma_rlp_are_isomorph}, we know
  $\cW\cap\Fib\sseq\cW\cap\{\Delta,\Gamma\}{_\perp}=\Iso$. 
\end{proof}

\begin{remark}
	In effect, we have shown that there is only one candidate for our desired model structure. Indeed, we have determined the weak equivalences and the trivial fibrations, so it follows that $\Cof={_\perp}(\cW\cap\Fib)$ and $\Fib=(\cW\cap\Cof)_{\perp}$ are uniquely determined as well, proving \autoref{thm:metric_MS_unique}.
\end{remark}

\subsection{Existence of the Metric model structure}
Our work so far suggests the following theorem:

\begin{theorem}\label{theorem:IoI_model_structure} There is a model structure on $\RCat$ (resp.\ $\bR_+\text-\Cat^\sym$) such that
	\begin{itemize}
		\item[($\cW$)] The weak equivalences are the fully faithful and essentially surjective $\bR_+$-functors.
		\item[($\Cof$)] Every $\bR_+$-functor is a cofibration.
		\item[($\Fib$)] The fibrations are those $\bR_+$ functors with the right lifting property against $\Delta$ and $\Gamma$.
	\end{itemize}
  Furthermore, the fibrant-cofibrant objects in this model structure are
  precisely the gaunt $\bR_+$-categories (resp.\ the symmetric gaunt
  $\bR_+$-categories). We call this model structure the \emph{Metric model
  structure}.
\end{theorem}

The characterization of the fibrant-cofibrant objects in this model structure
follows by
\autoref{remark:characterization_of_Delta_Gamma_lifting_properties}. This
subsection will be devoted to proving the existence of the Metric model
structure. To start, we prove an alternative characterization of the weak
equivalences that will be useful for the proof. To do so we first define a new
functor. Note that this is a functor between categories in the usual sense and
not an $\bR_+$-functor.

\begin{lemma}
	There exists a unique functor $M:\RCat\to\RCat^{\mathrm{gaunt}}$ satisfying the following properties:
	\begin{itemize}
		\item An $\bR_+$ category $\bC$ is sent to the gaunt category $M(\bC)$ whose objects are the isomorphism classes $[x]$ of objects $x$ in $\bC$, and the hom-objects are given by $M(\bC)([x],[y]) = \bC(x,y)$.
		\item An $\bR_+$ functor $f$ is sent to the function $[f]$ that sends $[x]$ to $[f(x)]$.
	\end{itemize}
\end{lemma}
\begin{proof}
	To see $M(\bC)$ is a well-defined $\bR_+$-category, we first need to show that the definition $M(\bC)([x],[y])=\bC(x,y)$ does not depend on the choice of representatives $x$ and $y$. Suppose $x'$ and $y'$ are objects in $\bC$ such that $\bC(x,x')=0=\bC(x',x)$ and $\bC(y,y')=0=\bC(y',y)$. Then by the triangle inequality (composition), we have
	\[\bC(x,y)\leq\bC(x,x')+\bC(x',y')+\bC(y',y)=\bC(x',y').\]
	By symmetry it follows $\bC(x,y)=\bC(x',y')$, so the hom-objects in $M(\bC)$ are well-defined, as desired. It is then straightforward to see that $M(\bC)$ is a gaunt $\bR_+$-category.

	To see that the action of $M$ on arrows is well-defined, note that if $f:\bC\to\bD$ is an $\bR_+$-functor and $x$ and $y$ are objects in $\bC$ with $\bC(x,y)=0=\bC(y,x)$, then by definition of an $\bR_+$-functor, we must have $\bC(f(x),f(y))=0$, so it follows $[f(x)]=[f(y)]$, hence $[f]$ is a well-defined map. Checking that $[f]$ is also $\bR_+$-functor and that $M$ is functorial is entirely straightforward.
\end{proof}

Using $M$, we may now provide an alternative characterization of the weak equivalences in the Metric model structure.

\begin{lemma}\label{lemma:M_char_of_metric_WE's}
  An $\bR_+$-functor $f:\bC\to\bD$ is fully faithful and essentially surjective
  (i.e., a weak equivalence in the Metric model structure) if and only if
  $M(f)$ is an isomorphism of $\bR_+$-categories.
\end{lemma}
\begin{proof}
	This follows immediately from unwinding the definitions.
\end{proof}

We now proceed to the proof of \autoref{theorem:IoI_model_structure}. 

\subsection*{Axiom CM1}
The fact that $\bR_+\text-\Cat$ (and $\bR_+\text-\Cat^\sym$) are (co)complete is \autoref{R+Cat_cocomplete}.

\subsection*{Axiom CM2}
That $\cW$ satisfies the 2-out-of-3 property follows from \autoref{lemma:M_char_of_metric_WE's} and the 2-out-of-3 property for isomorphisms.

\subsection*{Axiom CM3}
Now we want to show that $\Cof$, $\Fib$, and $\cW$ are closed under retracts. The class $\Cof$ is trivially closed under retracts as every $\bR_+$-functor is a cofibration. Since $\Fib=\{\Delta,\Gamma\}_\perp$ is characterized by a lifting property, it is closed under retracts (\cite[Lemma 11.1.4]{RiehlCatHmptyThry}). It remains to show that $\cW$ is closed under retracts. By applying $M$ to a retract diagram, this follows by functoriality of $M$ and \autoref{lemma:M_char_of_metric_WE's}, since isomorphisms in every category are closed under retracts.

\subsection*{Axiom CM4}\label{metric_MS_CM4}
Suppose we have a lifting problem in $\bR_+\text-\Cat$ of the following form
\[\begin{tikzcd}
	\bA & \bC \\
	\bB & \bD
	\arrow["f", from=1-1, to=1-2]
	\arrow["\iota"', from=1-1, to=2-1]
	\arrow["g", from=2-1, to=2-2]
	\arrow["\pi", from=1-2, to=2-2]
\end{tikzcd}\]
where $\iota$ is a cofibration and $\pi$ is a fibration. We claim that if either $\iota$ or $\pi$ is a weak equivalence then the diagram has a lift.

\textbf{Case 1:} First suppose $\iota\in\cW\cap\Cof$ and $\pi\in\Fib$. We define a lift $\ell:\bB\to\bC$. First, we fix some data: For each object $b$ in $\bB$, by essential surjectivity of $\iota$ there exists an object $a_b$ in $\bA$ such that $\iota(a_b)\cong b$. Then by the definition of an $\bR_+$-functor and commutativity of the lifting problem, we have that $\pi(f(a_b))=g(\iota(a_b))\cong g(b)$. Since $\pi$ has the right lifting property against $\Gamma$, by \autoref{remark:characterization_of_Delta_Gamma_lifting_properties}, there exists an object $c_b$ in $\bC$ such that $c_b\cong f(a_b)$ and $\pi(c_b)=g(b)$. In the case $b$ is in the image of $\iota$, we specifically choose $a_b$ such that $\iota(a_b)=b$ and set $c_b:=f(a_b)$.

Now we define $\ell(b):=c_b$. First of all, $\ell$ is an $\bR_+$-functor, as given objects $b$ and $b'$ in $\bB$, we have
\begin{align*}
  \bC(\ell(b),\ell(b'))
  =\bC(c_b,c_{b'})
  &\overset{(1)}=\bC(f(a_b),f(a_{b'})) \\
  &\overset{(2)}\leq\bA(a_b,a_{b'})
  \overset{(3)}=\bB(\iota(a_b),\iota(a_{b'}))
  \overset{(1)}=\bB(b,b'),
\end{align*}
where each occurrence of $(1)$ follows by an application of the triangle inequality and since $c_b\cong f(a_b)$ and $\iota(a_b)\cong b$, $(2)$ follows since $f$ is an $\bR_+$-functor, and $(3)$ follows since $\iota$ is a fully faithful $\bR_+$-functor. By \autoref{rmk:Rfunctorsobjects} it suffices to check commutativity on objects. It follows directly from the construction of $\ell$ that $\pi\circ\ell=g$. To show that $\ell\circ\iota=f$, let $a$ be an object in $\bA$, and set $b:=\iota(a)$. Then by unravelling definitions, we have
\[\ell(\iota(a))=\ell(b)=f(a_b).\]
Thus, it suffices to show $f(a)=f(a_b)$. To see this, note since $\iota(a)=b=\iota(a_b)$ and $\iota$ is fully faithful, $a$ and $a_b$ must be isomorphic. Hence since $f$ is an $\bR_+$-functor, we know $f(a)\cong f(a_b)$ in $\bC$. Furthermore, $\pi:\bC\to\bD$ is a fibration that maps $f(a)$ and $f(a_b)$ to the same point:
\[\pi(f(a_b))=g(\iota(a_b))=g(b)=g(\iota(a))=\pi(f(a)),\]
so by \autoref{remark:characterization_of_Delta_Gamma_lifting_properties}, we must have $f(a_b)=f(a)$, as desired.

\textbf{Case 2.} Next suppose $\iota\in\Cof$ and $\pi\in\cW\cap\Fib$. By \autoref{lemma:delta_gamma_rlp_are_isomorph}, $\pi$ is an isomorphism, meaning it has the right lifting property against every morphism.

\subsection*{Axiom CM5}
Let $f:\bC\to\bD$ be an $\bR_+$-functor. First, we prove $f$ factors as a trivial cofibration followed by a fibration. Define an $\bR_+$-category $\bL$ as follows: The objects of $\bL$ are pairs $([c],d)$, where $[c]$ is an isomorphism class of objects in $\bC$, and $d$ is an object in $\bD$ such that $f(c)\cong d$. This is well-defined, as if we have isomorphic objects $c$ and $c'$ in $\bC$ and some object $d$ in $\bD$ such that $f(c)\cong d$, then by the triangle inequality and the definition of an $\bR_+$-functor we have
\[\bD(f(c'),d)\leq\bD(f(c),d)+\bD(f(c'),f(c))=0+0=0,\]
and showing $\bD(d,f(c'))=0$ is analogous. Define the hom-objects in $\bL$ by 
\[\bL(([c],d),([c'],d)):=\bC(c,c').\] 
Showing this definition is well-defined and independent of choice of representatives follows similarly by the triangle inequality. Now, define $\iota:\bC\to\bL$ and $\pi:\bL\to\bD$ by $\iota(c)=([c],f(c))$ and $\pi([c],d)=d$.

First, we claim $\iota$ is a trivial cofibration, i.e., that $\iota$ is fully faithful and essentially surjective. It is fully faithful by construction, to see it is essentially surjective, let $([c],d)$ be an object in $\bD$. Then note
\[\bL(([c],d),\iota(c))=\bL(([c],d),([c],f(c)))=\bC(c,c)=0.\]
Showing $\bL(\iota(c),([c],d))=0$ is entirely analogous, so indeed $\iota(c)\cong([c],d)$. Thus, $\iota$ is an equivalence of $\bR_+$-categories as desired.

To check that $\pi$ is a fibration, we apply the characterization given by \autoref{remark:characterization_of_Delta_Gamma_lifting_properties}. First, suppose we have isomorphic objects $([c],d)$ and $([c'],d')$ in $\bL$ that are mapped to the same object by $\pi$, so $d=d'$. By the definition of hom-objects in $\cL$, since $([c],d)$ and $([c'],d')=([c'],d)$ are isomorphic we know $c$ and $c'$ are isomorphic in $\bC$, so $[c]=[c']$. Thus, $([c],d)=([c'],d')$. Next, let $([c],d)$ be an object in $\bL$, and let $d'$ be an object in $\bD$ such that $\pi([c],d)=d$ is isomorphic to $d'$. Then by the triangle inequality, we have
\[\bD(f(c),d')\leq\bD(f(c),d)+\bD(f(c),f(c))=0+0=0,\]
and showing $\bD(d',f(c))=0$ is analogous. Hence $f(c)\cong d'$, so $([c],d')$ is a valid element of $\bL$. Thus $d'=\pi([c],d')$ is in the image of $\pi$, as desired. Finally, by \ref{rmk:Rfunctorsobjects} it suffices to show $\pi\circ\iota=f$ on objects. This is clear since $\pi(\iota(c))=\pi([c],f(c))=f(c)$. Thus we've shown $f$ factors as a trivial cofibration followed by a fibration, as desired.

On the other hand, every $\bR_+$-functor $f$ factors as $\id\circ f$: Every $\bR_+$-functor is a cofibration, and identity morphisms are trivial fibrations by \autoref{lemma:delta_gamma_rlp_are_isomorph}.

\section{The Cauchy model structure on \texorpdfstring{$\RCat^\sym$}{R+-Cat{\^{}}sym}}\label{section:cauchy_model_structure}

In \Cref{section:Karoubian}, we defined the \emph{Karoubian model structure} on
$\Cat$. This model structure characterizes Cauchy complete categories, as
described in \Cref{section:preliminaries}. In this section, we construct an
$\bR_+$-enriched version of the Karoubian model structure on
$\bR_+\text-\Cat^\sym$: The cofibrations in both model structures are those ($\bR_+$-)functors that are injective on objects. The weak equivalences
in both model structures are the functors that induce equivalences between the
Cauchy completions. The fibrations in the Karoubian model structure are
characterized as those functors that ``lift retracts'' in a certain sense,
while the fibrations in the Cauchy model structure are characterized as those
$\bR_+$-functors that lift limits of Cauchy sequences. Finally, the
fibrant-cofibrant objects in both model structures are the Cauchy complete
($\bR_+$-) categories. Before we define the model structure, we first define an
important class of $\bR_+$-functors.

\begin{definition}\label{def:dense_R+_functor}
	Let $f:\bC\to\bD$ be an ${\bR_+}$-functor. Then we say that $f$ is \textit{dense} if for every object $d$ in $\bD$, there is some Cauchy sequence in the image of $f$ that converges to $d$.
\end{definition}

In comparison to the Karoubian model structure, one should think of dense $\bR_+$-functors as the $\bR_+$-enriched analog of the functors of categories that are surjective up to retracts (\autoref{def:surjective_up_to_retracts}). Now, we define an $\bR_+$-category that will be important for the construction of the Cauchy model structure.

\begin{definition}\label{def:Seq}
	Define a symmetric $\RR_+$-category $\Seq$ whose objects are the natural numbers $\bN=\{1,2,\ldots\}$, and for $n,m\in\NN$,
	\[
	\Seq(n, m) = \sum_{i=\min(n,m)}^{\max(n, m) - 1} \frac{1}{2^i}.
	\]
\end{definition}

One can easily see that $\Seq$ is a metric space. Intuitively, one should think of $\Seq$ as the natural numbers arranged in a line segment of length one. The numbers become increasingly bunched together as you go farther right along the segment. The distance from $n$ to $n+1$ is $1/2^n$, and the sequence $(n)_{n\in\bN}$ is in fact a Cauchy sequence in $\Seq$. In \autoref{lemma:convergent-Cauchy-const-or-inf}, we show that every Cauchy sequence in $\Seq$ is either eventually constant or is equivalent to the Cauchy sequence $(n)_{n\in\bN}$. Since $\Seq$ is a gaunt $\bR_+$-category, it follows that the objects of the Cauchy completion $\ol\Seq$ are
\[\ob\ol\Seq=\{\ell_{\wh 1},\ell_{\wh 2},\ldots\}\cup\{\ell_{(n)}\}\qquad\text{(\autoref{cor:description_of_seq_bar})}.\]
In other words, $\Seq$ is the $\bR_+$-category with a single nontrivial Cauchy
sequence (up to equivalence), and $\ol\Seq$ is the $\bR_+$-category obtained by
freely adding a limit to this Cauchy sequence. The inclusion
$\iota_\Seq:\Seq\into\ol\Seq$ takes an object $n$ in $\Seq$ to the presheaf
$\ell_{\wh{n}}=\bC(-,n)$ associated to the constant Cauchy sequence $\wh
n$.\footnote{We defined the notation $\ell_{(x_n)}$ for the presheaf
associated to a Cauchy sequence $(x_n)$ in
\autoref{def:presheaf_associated_to_cauchy_sequence}. Furthermore, recall we
characterized the Cauchy completion (\autoref{def:cauchy_completion}) $\ol\bC$
of a symmetric $\bR_+$-category $\bC$ as the full subcategory of
$\bR\text-\Fun(\bC^\op,\bR_+)$ on the presheaves of the form $\ell_{(x_n)}$
for a Cauchy sequence $(x_n)$ in $\bC$
(\autoref{prop:equivalent_Cauchy_characterization_fake}).} The $\bR_+$-functor
$\iota_\Seq$ will serve a similar role in the construction of the Cauchy model
structure to the role played by the inclusion $\Sigma:\Idem\into\Split$ in the
construction of the Karoubian model structure in
\Cref{section:Karoubian}. Explicitly, in
\autoref{prop:Seq_pick_out_subsequence}, we prove that given an
$\bR_+$-category $\bC$ containing a Cauchy sequence $(x_n)$, there exists an
$\bR_+$-functor $\nu:\Seq\to\bC$ such that the Cauchy sequence
$(\nu(n))_{n\in\bN}$ is a subsequence of $(x_n)$. Since sequences are
equivalent to their subsequences
(\autoref{lemma:equivalents-and-cauchy-coincide}), it follows that every Cauchy
sequence in an $\bR_+$-category may be represented (up to equivalence) by an
$\bR_+$-functor out of $\Seq$ (this is explicitly proven in
\autoref{prop:Seq_pick_out_subsequence}). 

Now we define the Cauchy model structure.

\begin{theorem}\label{theorem:cauchy_model_structure}
	There is a model structure on the category $\RCat^\sym$ of small ${\bR_+}$-enriched symmetric categories where \begin{itemize}
    \item[($\cW$)] The weak equivalences are the $\bR_+$-functors that are
    fully faithful and dense (\autoref{def:dense_R+_functor}), i.e., those
    $\bR_+$-functors $f:\bC\to\bD$ where $\ol f:\ol\bC\to\ol\bD$ is an
    equivalence (\autoref{prop:R_pastoral_equiv}).
		\item[($\Cof$)] The cofibrations are the ${\bR_+}$-functors that are injective on objects.
		\item[($\Fib$)] The fibrations are those ${\bR_+}$-functors with the right lifting property against $\iota_\Seq:\Seq\into\overline\Seq$.
	\end{itemize}
  Moreover, the fibrant-cofibrant objects in this model structure are
  precisely the Cauchy complete symmetric $\bR_+$-categories, i.e., those
  symmetric $\bR_+$-categories in which every Cauchy sequence has a limit. We
  call this the \emph{Cauchy model structure}.
\end{theorem}

Before giving the proof, we prove the following lemma which provides a useful characterization of the fibrations.

\begin{lemma}\label{lemma:characterization_of_fibrations_in_Cauchy_model_structure}
	Let $f:\bC\to\bD$ be an $\bR_+$-functor between symmetric $\bR_+$-categories. Then $f$ has the right lifting property against $\iota_\Seq$ (i.e., $f$ is a fibration in the Cauchy model structure) if and only if the following holds: Given a Cauchy sequence $(x_n)_{n\in\bN}$ in $\bC$ and a limit $y$ in $\bD$ of the Cauchy sequence $(f(x_n))_{n\in\bN}$, there exists an object $x$ in $\bC$ such that $x$ is a limit of $(x_n)$ and $f(x)=y$.
\end{lemma}
\begin{proof}
	Suppose $f$ is a vibration. By \autoref{prop:Seq_pick_out_subsequence}, there exists an $\bR_+$-functor $\nu:\Seq\to\bC$ such that the Cauchy sequence $(\nu(n))_{n\in\bN}$ is a subsequence of $(x_n)$. Then by \autoref{lemma:convergent_Cauchy_factors_through_seq_bar}, there exists an $\bR_+$-functor $\psi:\ol\Seq\to\bD$ sending $\ell_{(n)}\mapsto y$ that makes the following diagram commute:
	\[\begin{tikzcd}
		\Seq & \bC \\
		\ol\Seq & \bD
		\arrow["{\iota_\Seq}"', hook, from=1-1, to=2-1]
		\arrow["\nu", from=1-1, to=1-2]
		\arrow["f", from=1-2, to=2-2]
		\arrow["\psi", dashed, from=2-1, to=2-2]
	\end{tikzcd}\]
	Since $f$ is a fibration, the diagram has a lift $L:\ol\Seq\to\bC$. Set
	$x:=L(\ell_{(n)})$. By
	\autoref{lemma:convergent_Cauchy_factors_through_seq_bar} again, $x$ is a limit
	of the sequence $(\nu(n))_{n\in\bN}$. Then since $(\nu(n))$ is a subsequence of
	the Cauchy sequence $(x_n)$, $x$ must be a limit of $(x_n)$ as well
	(\autoref{lemma:equivalents-and-cauchy-coincide}). Finally, we have that
	$f(x)=f(L(\ell_{(n)}))=\psi(\ell_{(n)})=y$, as desired.

	On the other hand, suppose that $f$ satisfies the given condition and we have a lifting problem of the following form
	\[\begin{tikzcd}
		\Seq & \bC \\
		\ol\Seq & \bD
		\arrow["{\iota_\Seq}"', hook, from=1-1, to=2-1]
		\arrow["\psi", from=2-1, to=2-2]
		\arrow["\nu", from=1-1, to=1-2]
		\arrow["f", from=1-2, to=2-2]
	\end{tikzcd}\]
	Then $(\nu(n))$ is a Cauchy sequence, and its image under $f$ is $(f(\nu(n)))=(\psi(\iota_\Seq(n)))$. By \autoref{cor:lim_iota_cauchy_seq}, the object $\ell_{(n)}$ in $\ol\Seq$ is a limit of the sequence $(\iota_\Seq(n))$, so it follows that $\psi(\ell_{(n)})$ is a limit of $(f(\nu(n)))=(\psi(\iota_\Seq(n)))$. Then since $f$ satisfies the given condition, there exists an object $x$ in $\bC$ such that $x$ is a limit of $(\nu(n))$ and $f(x)=\psi(\ell_{(n)})$. Then by \autoref{lemma:convergent_Cauchy_factors_through_seq_bar}, there exists an $\bR_+$-functor $L:\ol\Seq\to\bC$ sending $\ell_{(n)}\mapsto x$ satisfying $L\circ\iota_\Seq=\nu$. To see $f\circ L=\psi$, first note by construction $f(L(\ell_{(n)}))=f(x)=\psi(\ell_{(n)})$. Moreover, given $n\in\bN$, we have 
	\[f(L(\ell_{\wh n}))=f(L(\iota_\Seq(n)))=f(\nu(n))=\psi(\iota_\Seq(n))=\psi(\ell_{\wh n}),\]
	as desired. Since $\ob\ol\Seq=\{\ell_{\wh 1},\ell_{\wh 2},\ldots\}\cup\{\ell_{(n)}\}$ (\autoref{cor:description_of_seq_bar}), it follows from \autoref{rmk:Rfunctorsobjects} that $f\circ L=\psi$, as desired.
\end{proof}

By this characterization of the fibrations, it easily follows that the fibrant-cofibrant objects in the Cauchy model structure are the Cauchy complete symmetric $\bR_+$-categories. We now proceed to the proof of the model structure. In what follows, we write $\cW$, $\Cof$, and $\Fib$ respectively to denote the classes of weak equivalences, cofibrations, and fibrations in the Cauchy model structure. 

\subsection*{Axiom CM1} The fact that $\bR\text-\Cat^\sym$ is (co)complete is \autoref{R+Cat_cocomplete}.

\subsection*{Axiom CM2}\label{cauchy_CM2} The 2-out-of-3 axiom for $\cW$ follows
from Axiom CM2 for the Metric model structure by passing to the induced maps
between Cauchy completions and applying \autoref{prop:R_pastoral_equiv}.

\subsection*{Axiom CM3}\label{cauchy_CM3} 
The fact that $\cW$ is closed under retracts follows from Axiom CM3 for the Metric model structure by passing to the induced maps between Cauchy completions and applying \autoref{prop:R_pastoral_equiv}. The fact that $\Cof$ is closed under retracts follows by the observation that $\Cof=\{\bJ\to\ast\}_\perp$, where $\bJ$ is the symmetric $\bR_+$-category with two objects $j_1$ and $j_2$ such that $\bJ(j_1,j_2)=\infty$. Thus, since $\Cof$ may be characterized by a lifting property, it is closed under retracts (see \cite[Lemma 11.1.4]{RiehlCatHmptyThry}). For the same reason, $\Fib$ is also closed under retracts.

\subsection*{Axiom CM4}\label{cauchy_CM4} Suppose we have a lifting problem in $\RCat^\sym$ of the following form,
\[\begin{tikzcd}
	\bA & \bC \\
	\bB & \bD
	\arrow["f", from=1-1, to=1-2]
	\arrow["\pi", two heads, from=1-2, to=2-2]
	\arrow["\iota"', tail, from=1-1, to=2-1]
	\arrow["g", from=2-1, to=2-2]
\end{tikzcd}\]
where $\iota$ is a cofibration and $\pi$ is a fibration. Then we want to show that if either $\iota$ or $\pi$ is a weak equivalence then the lifting problem has a solution $L:\bB\to\bC$. In both cases, we will construct the lift by first fixing some data associated to each object $b$ in $\bB$, and then we will define the lift using these data. If $b$ is in the image of $\iota$, we will need to be more careful about our choices, specifically so that the lift $L$ we define satisfies $L\circ\iota=f$.

\textbf{Case 1:} Suppose $\iota$ is a weak equivalence, so it is fully faithful and dense. We first fix the following data for each object $b$ in $\bB$.
\begin{itemize}
	\item Because $\iota$ is fully faithful and dense, we may choose a Cauchy sequence $\ol{a_b}:=(a_{b,n})_{n\in\bN}$ in $\bA$ such that the sequence $\iota(\ol{a_b})$ converges to $b$. If $b$ is in the image of $\iota$, since $\iota$ is injective, there exists a unique $a_b\in\bA$ such that $\iota(a_b)=b$. In this case, we specifically choose $\ol{a_b}$ to be $\wh{a_b}$, the constant Cauchy sequence on $a_b$.
	\item Since $b$ is a limit of $\iota(\ol{a_b})$, $g(b)$ is a limit of $\pi(f(\ol{a_b}))=g(\iota(\ol{a_b}))$. Thus since $\pi$ is a fibration, by \autoref{lemma:characterization_of_fibrations_in_Cauchy_model_structure}, we may fix an object $c_b$ in $\bC$ such that $c_b$ is a limit of $f(\ol{a_b})$ and $\pi(c_b)=g(b)$. When $b$ is in the image of $\iota$, we specifically set $c_b:=f(a_b)$, as in this case $f(\ol{a_b})=f(\wh{a_b})$ is the constant sequence on $f(a_b)$, which clearly has $f(a_b)$ as a limit, and moreover $\pi(f(a_b))=g(\iota(a_b))=g(b)$, as needed.
\end{itemize}

With this, define $L:\bB\to\bC$ by $L(b):=c_b$. We claim that $L$ is an $\bR_+$-functor. Let $b,d\in\bB$. Then since $c_b$ and $c_d$ are limits of $f(\ol{a_b})$ and $f(\ol{a_d})$, respectively, we have
\[
  \bC(L(b),L(d))
  =\bC(c_b,c_d)
  =\lim_{n\to\infty}\bC(f(a_{b,n}),f(a_{d,n})),
\]
where the second equality is \autoref{lemma:limits_can_be_moved_out_of_distance_function}. Now because $f$ is an $\bR_+$-functor and $\iota$ is fully faithful, for all $n$ we know that 
\[\bC(f(a_{b,n}),f(a_{d,n}))\leq\bA(a_{b,n},a_{d,n})=\bB(\iota(a_{b,n}),\iota(a_{d,n})).\] 
Therefore, we have that
\[\bC(L(b),L(d))=\lim_{n\to\infty}\bC(f(a_{b,n}),f(a_{d,n}))\leq\lim_{n\to\infty}\bB(\iota(a_{b,n}),\iota(a_{d,n}))=\bB(b,d),\]
where the inequality follows by \autoref{prop:lim_of_comparables}, and the last
equality follows again by
\autoref{lemma:limits_can_be_moved_out_of_distance_function} and the fact $b$
and $d$ are limits of the sequences $\iota(\overline{a_b})$ and
$\iota(\overline{a_d})$ respectively. Thus, $L$ is an $\bR_+$-functor, as
desired. By \autoref{rmk:Rfunctorsobjects} it suffices to show that $L$ defines a lift to the original
lifting problem (that it satisfies $L\circ\iota=f$ and $\pi\circ L=g$) on objects, which follows
directly from construction.

\textbf{Case 2:} Suppose that $\pi$ is a weak equivalence, so it is fully faithful and dense. The construction of the lift here is very similar to the construction of the lift in Case 1: For each object $b$ in $\bB$, because $\pi$ is fully faithful and dense, we may choose a Cauchy sequence $\ol{c_b}$ in $\bC$ such that $\pi(\ol{c_b})$ converges to $g(b)$. If $b$ is in the image of $\iota$, then we specifically choose $\ol{c_b}$ to be the constant Cauchy sequence on $f(a_b)$, where $a_b$ is the unique object in $\bA$ satisfying $\iota(a_b)=b$. Then since $\pi$ is a fibration, by \autoref{lemma:characterization_of_fibrations_in_Cauchy_model_structure} there exists an object $c_b'$ in $\bC$ such that $c_b'$ is a limit of $\ol{c_b}$ and $\pi(c_b')=g(b)$. In the case $b$ is in the image of $\iota$, we specifically set $c_b':=f(a_b)$. Finally, we may define $L(b):=c_b'$. That $L$ is an $\bR_+$-functor follows from the fact that $g$ is an $\bR_+$-functor and $\pi$ is fully faithful. That $L$ is a solution to the lifting problem follows directly from construction.

\subsection*{Axiom CM5}\label{cauchy_CM5} Let $f:\bC\to\bD$ be an $\bR_+$-functor between symmetric $\bR_+$-categories. 

\textbf{Case 1:} First, we prove that $f$ may be factored as a trivial cofibration followed by a fibration. To start, we define a symmetric $\bR_+$-category $\bL$: The objects of $\bL$ are pairs $((x_n),y)$ where $(x_n)$ is a Cauchy sequence in $\bC$ and $y\in\bD$ is a limit of the Cauchy sequence $(f(x_n))$. Then we define
\[\bL(((x_n),y),((x_n'),y')):=\ol\bC(\ell_{(x_n)},\ell_{(x_n')})=\lim_{n\to\infty}\bC(x_n,x_n'),\]
where the second equality follows by \autoref{prop:Lmetric-on-cauchy-completion}. It is straightforward to verify that $\bL$ is a symmetric $\bR_+$-category since $\ol\bC$ is.

Next, define $\iota:\bC\to\bL$ by $\iota(c):=(\wh c,f(c))$, where $\wh c$ denotes the constant Cauchy sequence on $c$. 
It is clear that $\iota$ is injective on objects and fully faithful by construction. To show that $\iota$ is dense, 
let $((x_n),y)$ be an object in $\bL$. We wish to show that $((x_n),y)$ is a limit of the Cauchy sequence $(\iota(x_n))$. This is clear, as
\[\lim_{m\to\infty}\bL(((x_n),y),\iota(x_m))=\lim_{m\to\infty}\ol\bC(\ell_{(x_n)},\ell_{\wh{x_m}})=0,\]
where the last equality follows by \autoref{cor:lim_iota_cauchy_seq}, that
tells us $\ell_{(x_n)}$ is a limit of the sequence
$(\ell_{\wh{x_n}})_{n\in\bN}$. Thus we've shown $\iota$ is injective on
objects, fully faithful, and dense, so $\iota$ is a trivial cofibration as
desired.

Finally, we will define $\pi:\bL\to\bD$ and show that it is a fibration satisfying $\pi\circ\iota=F$. Given an object $((x_n),y)$ in $\bL$, we define $\pi((x_n),y):=y$. To see that $\pi$ is an $\bR_+$-functor, let $((x_n),y)$ and $((x_n'),y')$ be objects in $\bL$. Then since $y$ is a limit of $(f(x_n))$ and $y'$ is a limit of $(f(x_n'))$, by \autoref{lemma:limits_can_be_moved_out_of_distance_function}, we have:
\[\bD(\pi((x_n),y),\pi((x_n'),y'))=\bD(y,y')=\lim_{n\to\infty}\bD(f(x_n),f(x_n')).\]
Then because $f$ is an $\bR_+$-functor, by \autoref{prop:lim_of_comparables} we have:
\[\lim_{n\to\infty}\bD(f(x_n),f(x_n'))\leq\lim_{n\to\infty}\bC(x_n,x_n')=\bL(((x_n),y),((x_n'),y')).\]

It remains to show that $\pi$ is a fibration. We use the characterization afforded by \autoref{lemma:characterization_of_fibrations_in_Cauchy_model_structure}: Suppose $(z_n)_{n\in\bN}$ is a Cauchy sequence in $\bL$ and $d\in\bD$ is a limit of the sequence $(\pi(z_n))_{n\in\bN}$ in $\bD$. Then we wish to find an element $z$ in $\bL$ such that $z$ is a limit of $(z_n)$ and $\pi(z)=d$. For each $n\in\bN$, write $z_n=(\ol{x_n},y_n)$, where $\ol{x_n}$ is shorthand for a Cauchy sequence $(x_{n,m})_{m\in\bN}$ in $\bC$. Note that $(\ell_{\ol{x_n}})_{n\in\bN}$ is a Cauchy sequence as
\[\ol\bC(\ell_{\ol{x_n}},\ell_{\ol{x_m}})=\ol\bL(z_n,z_m),\]
and $(z_n)_{n\in\bN}$ is Cauchy in $\bL$ by definition. Thus, since $\ol\bC$ is
Cauchy complete (\autoref{prop:R-cauchy-complete-is-cauchy-complete}), there
exists a Cauchy sequence $\ol a=(a_n)$ in $\bC$ such that $\ell_{\ol a}$ is a
limit of $(\ell_{\ol{x_n}})$ in $\ol\bC$. Then, we claim $(\ol a,d)$ is a valid
element of $\bL$, i.e., that $d$ is a limit of $f(\ol a)$. If this is true, we are done, as $\pi(\ol a,d)=z$. By definition
\[\bL((\ol a,d),z_n)=\ol\bC(\ell_{\ol a},\ell_{\ol{x_n}}),\]
 and by construction $\ell_{\ol a}$ is a limit of $(\ell_{\ol{x_n}})_{n\in\bN}$. It follows that $(\ol a,d)$ is a limit of $(z_n)$.

Now, we prove that $d$ is a limit of $f(\ol a)$. Unfortunately, what follows is rather technical. First of all, recall $d$ is a limit of $(\pi(z_n))=\ol y$, where $\ol y$ denotes the sequence $(y_n)_{n\in\bN}$. We claim $\ell_{\ol y}$ is a limit of the sequence $(\ell_{f(\ol{x_n})})_{n\in\bN}$ in $\ol\bD$. Since $y_n$ is a limit of $f(\ol{x_n})$ and $\iota_\bD$ is an $\bR_+$-functor, we have that $\iota_\bD(y_n)=\ell_{\wh{y_n}}$ is a limit of $\iota_\bD(f(\ol{x_n}))$ in $\ol\bD$. We also know $\ell_{f(\ol{x_n})}$ is a limit of $\iota_\bD(f(\ol{x_n}))$ by \autoref{cor:lim_iota_cauchy_seq}. Since limits in $\ol\bD$ are unique (by \autoref{cor:R-Cauchy_completion_is_gaunt}), it follows that $\ell_{f(\ol{x_n})}=\ell_{\wh{y_n}}$. Hence, by \autoref{cor:lim_iota_cauchy_seq} again, we have that $\ell_{\ol y}$ is a limit of $(\ell_{\wh{y_n}})=(\ell_{f(\ol{x_n})})$, as desired. Finally, note that since $\ol f$ is an $\bR_+$-functor (\autoref{prop:induced_R-functor_between_completions}) and $\ell_{\ol a}$ is a limit of $(\ell_{\ol{x_n}})$, we have that $\ell_{f(\ol a)}=\ol f(\ell_{\ol a})$ is a limit of $(\ell_{f(\ol{x_n})})_{n\in\bN}=(\ol f(\ell_{x_n}))_{n\in\bN}$ as well. Again since limits in $\ol\bD$ are unique, it follows that $\ell_{f(\ol a)}=\ell_{\ol y}$, so the sequences $f(\ol a)$ and $\ol y$ are equivalent, and thus have the same limit by \autoref{lemma:equivalents-and-cauchy-coincide}. Hence, $d$ is a limit of $f(\ol a)$, as desired.

Thus, we have constructed a trivial cofibration $\iota:\bC\to\bL$ and a fibration $\pi:\bL\to\bD$. Finally, given an object $c$ in $\bC$, we have $\pi(\iota(c))=\pi(\wh c,f(c))=f(c)$. Then, by \autoref{rmk:Rfunctorsobjects}, $\pi\circ\iota=f$ as desired.

\textbf{Case 2:} On the other hand, we claim $f$ factors as a cofibration followed by a trivial fibration. We define a new category $\bL$ as follows: The objects of $\bL$ are defined by $\ob(\LL)=\ob(\CC)\amalg\Ob(\DD)$. We define the hom-objects as follows: 
\begin{equation*}
	\LL(x, y) =
	\begin{cases}
		\DD(f(x), y) &\mathrm{if }\, x\in\ob(\CC)\text{ and }y\in\ob(\DD)\\
		\DD(x, f(y)) &\mathrm{if }\, x\in\ob(\DD)\text{ and }y\in\ob(\CC)\\
		\DD(f(x), f(y)) &\mathrm{if }\, x,y\in\ob(\CC)\\
		\DD(x, y) &\mathrm{if }\, x,y\in\ob(\DD).
	\end{cases}
\end{equation*}
By construction, $\bL$ is a symmetric $\bR_+$-category. We now define
$\RR_+$-functors $\iota:\CC\to\LL$ and $\pi:\LL\to\DD$ as follows: 
\begin{itemize}
	\item $\iota(c) = c$.
	\item $\pi(x) = \begin{cases}
		f(x) &\mathrm{if }\, x\in\ob(\CC)\\
		x &\mathrm{if }\, x\in\ob(\DD)
	\end{cases}$
\end{itemize}
By construction $\iota$ is an $\bR_+$-functor that is injective on objects and $\pi$ is a fully faithful \textit{strictly} surjective $\bR_+$-functor. Moreover, we clearly have $\pi\circ\iota=f$. It remains to show that $\pi$ is a fibration. Again, we use the characterization of fibrations given in \autoref{lemma:characterization_of_fibrations_in_Cauchy_model_structure}. Suppose we have a Cauchy sequence $(x_n)$ in $\bL$ and an object $d$ in $\bD$ such that $d$ is a limit of $(\pi(x_n))$. It follows directly from construction that $d$, viewed as an element of $\ob\bL\spseq\ob\bD$, is a limit of $(x_n)$. Moreover, $\pi(d)=d$. Thus, $\pi$ is a fibration as desired.

\section{The Cauchy-Metric model structure on \texorpdfstring{$\RCat^\sym$}{R+-Cat{\^{}}sym}}\label{section:cauchy_metric_model_structure}

So far, we have constructed two model structures on $\mathbb{R}\text-\Cat^\sym$: the Metric and the Cauchy model structures. In this section, we construct a model structure that serves as a combination of the two: the fibrant-cofibrant objects are the Cauchy complete and gaunt $\mathbb{R}$-categories.

\begin{theorem}\label{theorem:cauchy_metric_model_structure}
	There is a model structure on $\bR_+\text-\Cat^\sym$ where:
	\begin{itemize}
    \item[($\cW$)] The weak equivalences $\cW$ are given by those
    $\bR_+$-functors that are fully faithful (an isometry) and dense
    (\autoref{def:dense_R+_functor}). 
		\item[($\Cof$)] Every $\bR_+$-functor is a cofibration.
		\item[($\Fib$)] The fibrations are the $\bR_+$-functors that have the right lifting property with respect to $\Delta$, $\Gamma$, and $\iota_\Seq$, i.e., $\Fib:=\{\Delta,\Gamma,\iota_\Seq\}_\perp$.
	\end{itemize}
  Moreover, the fibrant-cofibrant objects in this model structure are the gaunt Cauchy-complete $\mathbb{R}$-categories. We call this the \textit{Cauchy-metric} model structure on $\mathbb{R}\text-\Cat^\sym$. 
\end{theorem}

Note the fibrations in the Cauchy-Metric model structure are precisely the
morphisms that are fibrations in both the Metric and the Cauchy model
structures on $\mathbb{R}\text-\Cat^\sym$. Thus, the characterization of
the fibrant-cofibrant objects in the Cauchy-Metric model structure follows from
the characterization of the fibrant-cofibrant objects in the Metric and Cauchy
model structures. Before proving the rest of the theorem, we prove the
following propositions that will be useful for the proof. We
will write $\mathcal{W}$, $\Cof$, and $\Fib$ to denote respectively, the weak equivalences,
cofibrations, and fibrations in the Cauchy-Metric model structure.

\begin{proposition}\label{prop:trivial_fibrations_in_combined_are_isomorphisms}
  If $\Iso$ denotes the class of isomorphisms in $\mathbb{R}\text-\Cat^\sym$, then $\mathcal{W}\cap\Fib\subseteq\Iso$.
\end{proposition}
\begin{proof}
  By \autoref{lemma:delta_gamma_rlp_are_isomorph}, it suffices to show that
  every trivial fibration in the Cauchy-Metric model structure is a trivial
  fibration in the Metric model structure. If we write $\Cof_M$, $\Fib_M$, and
  $\mathcal{W}_M$ for the cofibrations, fibrations, and weak equivalences in
  the Metric model structure on $\mathbb{R}_+\text-\Cat^\sym$
  (\autoref{theorem:IoI_model_structure}), then clearly $\Fib\subseteq\Fib_M$,
  so it further suffices to show that
  $\mathcal{W}\cap\Fib\subseteq\mathcal{W}_M$. 

  Let $f\in \mathcal{W}\cap\Fib$ and let $d\in \bD$. Since $f$ is dense there
  exists a Cauchy sequence $(f(x_n))_{n\in \bN}$ that converges to $d$. Because
  $f$ is fully faithful $(x_n)_{n\in \bN}$ is also Cauchy
  (\autoref{prop:R-functors-preserve-everything}). By
  \autoref{lemma:characterization_of_fibrations_in_Cauchy_model_structure},
  there exists $c\in \bC$ such that $f(c) = d$ so $f$ is surjective and thus
  essentially surjective. The desired result follows.
\end{proof}

While the following remark is not integral to the proofs in this section, it provides some helpful intuition.

\begin{remark}\label{prop:fib_Delta_iota}
Standard model categorical techniques can be used to show that $\{\Gamma,\Delta,\iota_\Seq\}_\perp=\{\Delta,\iota_\Seq\}_\perp$.  
\end{remark}


Finally, before giving the proof of
\autoref{theorem:cauchy_metric_model_structure}, we prove a useful lemma.

\begin{lemma}\label{ol_f_is_fibration_in_CM}
  Let $f:\mathbb{C}\to\mathbb{D}$ be a morphism in
  $\mathbb{R}_+\text-\Cat^\sym$, then the induced $\mathbb{R}_+$-functor on
  Cauchy completions $\overline{f}:\overline{\mathbb{C}}\to
  \overline{\mathbb{D}}$ constructed in
  \autoref{prop:induced_R-functor_between_completions} is a fibration in the
  Cauchy-Metric model structure, i.e., it belongs to
  $\left\{\Delta,\Gamma,\iota_\Seq\right\}_\perp$.
\end{lemma}
\begin{proof}
  By \autoref{cor:R-Cauchy_completion_is_gaunt}, $\overline{\mathbb{C}}$ and
  $\overline{\mathbb{D}}$ are gaunt, so clearly $f$ belongs to
  $\{\Delta,\Gamma\}_\perp$, by the characterization given in
  \autoref{remark:characterization_of_Delta_Gamma_lifting_properties}. Thus, it
  remains to show that $\overline{f}$ has the right lifting property against
  $\iota_\Seq:\Seq\to \overline{\Seq}$. To this end, we use the
  characterization afforded by
  \autoref{lemma:characterization_of_fibrations_in_Cauchy_model_structure}: Let
  $(x_n)$ be a Cauchy sequence in $\overline{\mathbb{C}}$ and let $y\in
  \overline{\mathbb{D}}$ be a limit of $(f(x_n))$. Since
  $\overline{\mathbb{C}}$ is Cauchy complete
  (\autoref{prop:R-cauchy-complete-is-cauchy-complete}), the sequence $(x_n)$
  has a limit, call it $x$. Then by
  \autoref{prop:R-functors-preserve-everything}, $f(x)$ is a limit of
  $(f(x_n))$. Finally, since $f(x)$ and $y$ are both limits of $(f(x_n))$ in
  $\overline{\mathbb{D}}$, that is gaunt, we have that $f(x)=y$ by
  \autoref{prop:isomorphic_limits}, as desired.
\end{proof}

We will now prove \autoref{theorem:cauchy_metric_model_structure}.

\subsection*{Axiom CM1} The fact that $\bR\text-\Cat^\sym$ is (co)complete is \autoref{R+Cat_cocomplete}.

\subsection*{Axiom CM2} 
This model structure has the same weak equivalences as the Cauchy model
structure (\autoref{theorem:cauchy_model_structure}), so the weak equivalences
are closed under 2-of-3.

\subsection*{Axiom CM3} Again, since the Metric and Cauchy-Metric model
structures have the same weak equivalences, $\mathcal{W}$ is
closed under retracts. The cofibrations are closed under
retracts, as every morphism is a cofibration. Finally, $\Fib$ is closed under
retracts as it is characterized by a lifting property.

\subsection*{Axiom CM4} Suppose we have a lifting problem in $\bR_+\text-\Cat^\sym$ of the following form,
\[\begin{tikzcd}
	\bA & \bC \\
	\bB & \bD
	\arrow["\iota"', tail, from=1-1, to=2-1]
	\arrow["g", from=2-1, to=2-2]
	\arrow["f", from=1-1, to=1-2]
	\arrow["\pi", two heads, from=1-2, to=2-2]
\end{tikzcd}\]
where $\iota$ is a cofibration and $\pi$ is a fibration.  We claim that if either is a weak equivalence the lifting problem has a solution.

\textbf{Case 1:} Suppose $\iota$ is a weak equivalence. The proof of this case
is very similar to the proof given for CM4 of the Metric model structure
(\autoref{metric_MS_CM4}), but there are some slight differences. First, we fix
some data: For each object $b$ in $\mathbb{B}$, since $\iota$ is dense, there
exists a sequence $\overline{a_b}=(a_{b,n})_{n\in\mathbb{N}}$ in $\mathbb{A}$
such that $\iota(\overline{a_b})$ is Cauchy with limit $b$. 
Since $\iota$ is fully faithful, \autoref{prop:R-functors-preserve-everything}
yields that $\overline{a_b}$ is a Cauchy sequence in $\mathbb{A}$, and
therefore $f(\overline{a_b})$ is a Cauchy sequence in $\mathbb{C}$. Moreover,
$g(b)$ is a limit of $\pi(f(\overline{a_b}))=g(\iota(\overline{a_b}))$, and
$\pi$ has the right lifting property against $\iota_\Seq$, so by
\autoref{lemma:characterization_of_fibrations_in_Cauchy_model_structure} the
sequence $f(\overline{a_b})$ has a limit $c_b$ such that $\pi(c_b)=g(b)$.
In the case $b$ is in the image of $\iota$, we specifically choose $a_b$ such
that $\iota(a_b)=b$, and set $\overline{a_b}=\widehat{a_b}$ to be the constant
sequence on $a_b$, and $c_b:=f(a_b)$.

Now, we define $\ell(b):=c_b$. First of all, $\ell$ is an
$\mathbb{R}_+$-functor, as given objects $b$ and $b'$ in $\mathbb{B}$, we have
\begin{align*}
  \mathbb{C}(\ell(b),\ell(b'))
  &=\mathbb{C}(c_b,c_{b'}) \\
  &\overset{(1)}=\lim_{n\to\infty}\mathbb{C}(f(a_{b,n}),f(a_{b',n})) \\
  &\overset{(2)}\leq \lim_{n\to\infty}\mathbb{A}(a_{b,n},a_{b',n}) \\
  &\overset{(3)}=\lim_{n\to\infty}\mathbb{B}(\iota(a_{b,n}),\iota(a_{b',n})) \\
  &\overset{(1)}=\mathbb{B}(b,b'),
\end{align*}
where each occurrence of (1) follows by
\autoref{lemma:limits_can_be_moved_out_of_distance_function} and because $c_b$
is a limit of $f(\overline{a_b})$ and $b$ is a limit of
$\iota(\overline{a_b})$, (2) follows since $f$ is an $\mathbb{R}_+$-functor,
and (3) follows because $\iota$ is fully faithful. By \autoref{rmk:Rfunctorsobjects} it suffices to check commutativity on objects. It then follows directly from the construction of $\ell$ that $\pi\circ\ell=g$. To show that $\ell\circ\iota=f$, let $a$ be an object in $\mathbb{A}$, and set $b:=\iota(a)$. Then by unravelling definitions, we have
\[
  \ell(\iota(a))
  =\ell(b)
  =c_b.
\]
Thus, it suffices to show $f(a)=c_b$. To see this, note since $\iota(a)=b$ is a
limit of the sequence $\overline{a_b}$ and $\iota$ is fully faithful, $a$ must
be a limit of $\overline{a_b}$ (\autoref{prop:R-functors-preserve-everything}),
so that $f(a)$ is a limit of $\iota(\overline{a_b})$. Then $f(a)\cong c_b$ by
\autoref{prop:isomorphic_limits}, since $c_b$ and $f(a)$ are both limits of
$\iota(\overline{a_b})$. Moreover, we have by definition that
\[
  \pi(c_b)
  =g(b)
  =g(\iota(a))
  =\pi(f(a)).
\]
Then since $\pi$ belongs to ${\{\Delta,\Gamma\}}_\perp$ and sends the
isomorphic objects $f(a)$ and $c_b$ to the same element in $\mathbb{D}$, it
follows by \autoref{remark:characterization_of_Delta_Gamma_lifting_properties}
that $f(a)$ and $c_b$ must be equal in the first place, as desired.

\textbf{Case 2:} Suppose $\pi$ is a weak equivalence. By \autoref{prop:trivial_fibrations_in_combined_are_isomorphisms}, $\pi$ is an isomorphism, 
so a lift exists.

\subsection*{Axiom CM5} Let $f:\bC\to\bD$ be an $\bR_+$-functor between symmetric $\bR_+$-categories. 

\textbf{Case 1:} We wish to construct a symmetric $\bR_+$-category $\bL$
and $\bR_+$-functors $\iota: \bC \to \bL$ and $\pi: \bL \to \bD$ such that
$\iota$ is a trivial cofibration, $\pi$ is a fibration, and $\pi \circ \iota =
f$. To that end, define $\mathbb{L}$, $\iota$, $\pi$, and $g:\mathbb{L}\to
\overline{\mathbb{C}}$ to fit into the following pullback diagram in
$\mathbb{R}_+\text-\Cat^\sym$
\[\begin{tikzcd}
	{\mathbb{C}} \\
	& {\mathbb{L}} & {\mathbb{D}} \\
	& {\overline{\mathbb{C}}} & {\overline{\mathbb{D}}}
	\arrow["\iota", dashed, from=1-1, to=2-2]
	\arrow["f", curve={height=-12pt}, from=1-1, to=2-3]
	\arrow["{\iota_{\mathbb{C}}}"', curve={height=12pt}, from=1-1, to=3-2]
	\arrow["\pi", from=2-2, to=2-3]
	\arrow["g"', from=2-2, to=3-2]
	\arrow["\lrcorner"{anchor=center, pos=0.125}, draw=none, from=2-2, to=3-3]
	\arrow["{\iota_{\mathbb{D}}}", from=2-3, to=3-3]
	\arrow["{\overline{f}}", from=3-2, to=3-3]
\end{tikzcd}\]
where $\overline{f}$ is constructed and proven to satisfy
$\iota_\mathbb{D}\circ f=\overline{f}\circ\iota_\mathbb{C}$ in
\autoref{prop:induced_R-functor_between_completions}.
We know $\pi\circ\iota=f$ by definition, so it suffices to show that $\pi$ is a
fibration and $\iota$ is a trivial cofibration.

To see $\pi$ is a fibration, first note that $\Fib$ is characterized by a right
lifting property, meaning it is closed under pullbacks (this is a dual version
of the fact that a class of morphisms characterized by a left-lifting property
is closed under pushouts, see \cite[Lemma 11.1.4]{RiehlCatHmptyThry}). Thus
$\pi$ is a fibration by \autoref{ol_f_is_fibration_in_CM}.

To see $\iota$ is a trivial cofibration, we need to show it is fully faithful
and dense. To see it is fully faithful, first note that
$\mathbb{L}(\iota(c),\iota(c'))\leq\mathbb{C}(c,c')$ since $\iota$ is an
$\mathbb{R}_+$-functor, and on the other hand, because $\iota_\mathbb{C}$ is
fully faithful and $g$ is an $\mathbb{R}_+$-functor, we have
\[
  \mathbb{C}(c,c')
  =\overline{\mathbb{C}}(\iota_{\mathbb{C}}(c),\iota_{\mathbb{C}}(c'))
  =\overline{\mathbb{C}}(g(\iota(c)),g(\iota(c')))
  \leq\mathbb{L}(\iota(c),\iota(c')).
\]
Thus $\mathbb{C}(c,c')=\mathbb{L}(\iota(c),\iota(c'))$, so $\iota$ is fully
faithful, as desired. To see $\iota$ is dense, let $\ell\in\mathbb{L}$. Since $\iota_{\mathbb{C}}$ is dense
(\autoref{prop:equivalent_Cauchy_characterization}) and fully faithful, there
exists a Cauchy sequence $(x_n)$ in $\mathbb{C}$ such that $g(\ell)$ is a limit
of $(\iota_{\mathbb{C}}(x_n))$. Thus
$\iota_{\mathbb{D}}(\pi(\ell))=\overline{f}(g(\ell))$ is a limit of
$(\iota_{\mathbb{D}}(f(x_n)))=(\overline{f}(\iota_\mathbb{C}(x_n)))$ by
\autoref{prop:R-functors-preserve-everything}. Then because
$\iota_{\mathbb{D}}$ is fully faithful, it reflects limits
(\autoref{prop:R-functors-preserve-everything}), so $\pi(\ell)$ is a limit of
$(f(x_n))=(\pi(\iota(x_n)))$. Now, since $\pi$ is a fibration, it has the right
lifting property against $\iota_{\Seq}$, so by
\autoref{lemma:characterization_of_fibrations_in_Cauchy_model_structure} there
exists some $\ell'\in\mathbb{L}$ such that $\ell'$ is a limit of $(\iota(x_n))$
and $\pi(\ell')=\pi(\ell)$. Moreover, note that $g(\ell)$ and $g(\ell')$ are
both limits of $(g(\iota(x_n)))$ in $\overline{\mathbb{C}}$, and
$\overline{\mathbb{C}}$ is gaunt, so that $g(\ell)=g(\ell')$
(\autoref{prop:isomorphic_limits}). It follows purely formally that
$\ell=\ell'$ since $g(\ell)=g(\ell')$ and $\pi(\ell)=\pi(\ell')$.%
\footnote{
  Suppose for the sake of a contradiction that $\ell\neq\ell'$. Then consider
  the maps $p:\ast\to \overline{\mathbb{C}}$ and $q:\ast\to \mathbb{D}$ picking
  out $g(\ell)$ and $\pi(\ell)$, respectively. Clearly $\overline{f}\circ
  p=\iota_{\mathbb{D}}\circ q$, so by the universal property of the pullback
  there exists a \emph{unique} arrow $h:\ast\to\mathbb{L}$ satisfying $g\circ
  h=p$ and $\pi\circ h=q$. Yet we can construct two distinct maps
  $\ast\to\mathbb{L}$ satisfying these equalities, namely, the maps picking out
  either $\ell$ or $\ell'$. We reach a contradiction, we must have had that
  $\ell=\ell'$ in the first place.
}
Thus $\ell$ is a limit of $(\iota(x_n))$, as desired, meaning $\iota$ is indeed
dense.

\textbf{Case 2:}  On the other hand, since every morphism is a cofibration and
every isomorphism is a trivial fibration, we have that $f=\id_{\mathbb{D}}\circ
f$ factors as a cofibration followed by a trivial fibration, as desired.

\begin{proposition}\label{lemma:cauchy_metric_uniqueness}
  The Cauchy-Metric model structure
  (\autoref{theorem:cauchy_metric_model_structure}) is the \emph{unique} model
  structure on $\RCat^{\text{sym}}$ where
	\begin{enumerate}
		\item The weak equivalences are the fully faithful and dense, and 
    \item Every fibrant-cofibrant object is gaunt.
	\end{enumerate}
\end{proposition}
\begin{proof}
  The proof is very similar to that of \autoref{thm:metric_MS_unique}, so we
  only give a sketch here. To start, one fixes a model structure $\mathcal{M}$
  on $\mathbb{R}_+\text-\Cat^\sym$ with fibrations $\Fib$, cofibrations $\Cof$,
  and weak equivalences $\mathcal{W}$ satisfying (1) and (2) above. Essentially
  the exact same argument given in Lemmas \ref{lemma:0_to_b_cof} and
  \ref{cor:inj_is_cof} (using that weak equivalences are dense, so the domain
  of a weak equivalence cannot be empty) yields that every injective
  $\mathbb{R}_+$-functor is a cofibration in $\mathcal{M}$. Then one can
  further show that the class of cofibrations cannot be the class of injections: If
  this was the case, then you could directly show that the category
  $\mathbb{I}$ (\autoref{I_Delta_Gamma_defns}) is both fibrant and cofibrant,
  contradicting the fact that the fibrant-cofibrant objects are gaunt. Then the
  same argument given in
  \autoref{Delta_Gamma_are_trivial_cofibrations_uniqueness_proof_lemma} would
  yield that $\Delta$ and $\Gamma$ are trivial cofibrations in $\mathcal{M}$,
  and clearly, $\iota_\Seq$ would be a trivial cofibration as well by the
  characterization of $\overline{\Seq}$ given following \autoref{def:Seq}. Thus
  every fibration in $\mathcal{M}$ would be a fibration in the Cauchy-Metric
  model structure, so by
  \autoref{prop:trivial_fibrations_in_combined_are_isomorphisms}, we would have
  that $\mathcal{W}\cap\Fib \subseteq\Iso$. Yet we also know that the
  isomorphisms lift against every morphism, so we'd have that
  $\Iso\subseteq{(\Cof)}_\perp=\mathcal{W}\cap\Fib$. Thus the fibrations and
  weak equivalences have been uniquely determined; the desired result follows.
\end{proof}

\begin{remark}\label{rmk:localizations}
	It is worth briefly mentioning the relation of our three model structures to one another, working on $\RCat^\sym$. We will write $(\RCat^\sym)_{\on{Metric}}$, $(\RCat^\sym)_{\on{Cauchy}}$, and $(\RCat^\sym)_{\on{MC}}$ for the metric, Cauchy, and Cauchy-Metric model structures, respectively. The identity functor then defines a left Quillen functor
	\[
	\begin{tikzcd}
		(\RCat^\sym)_{\on{Metric}}\arrow[r] &(\RCat^\sym)_{\on{MC}}
	\end{tikzcd}
	\]  
	Since the sets of cofibrations in the two model structure coincide, we can go further, and note that this is a left Bousfield localization.
	
	On the other hand, we can note that the identity functor on $\RCat^\sym$ also defines a left Quillen functor
	\[
	\begin{tikzcd}
		(\RCat^\sym)_{\on{Cauchy}}\arrow[r] & (\RCat^\sym)_{\on{MC}}.
	\end{tikzcd}
	\]
	Though this is not a left Bousfield localization.  
\end{remark}

\appendix

\section{Analysis in \texorpdfstring{$\bR_+\text-\Cat$}{R+-Cat}}\label{app:analysis}

The appendix collects technical results on limits and Cauchy sequences in Lawvere metric spaces. Much of this material closely parallels the analysis of metric spaces, but because of key distinctions, care must be taken in generalizing results to this setting. The basic arithmetic operations ($+$, $-$, etc.) use different conventions in $\RR_+$-categories than in more conventional metric spaces. As such, many results are proven in cases, depending on whether distances are infinite. Note also that the limit of a sequence in $\RR_+$ is only ever $\infty$ when that sequence eventually becomes constant on $\infty$. The proofs of many of the statements below are entirely similar to their metric-space analogs or are straightforward, in which case no proof is given.

Some of the analysis of limits in Lawvere metric spaces is considered in the papers \cite{Rutten} and \cite{Bon}, where the authors consider directional versions of limits in non-symmetric Lawvere metric spaces. As such, while there are connections between these works and the results of this appendix, there does not appear to be any redundancy.

\begin{lemma}\label{lemma:arithmetic_in_R+}
	Let $a,b,c\in \RR_+$. Then we have the following: 
	\begin{enumerate}[label=(\roman*)]
		\item $(a+b)-a\leq b$.
		\item $|(a-b)-c|\leq|a-c|+|b|$, in particular, $|a-b|\leq|a|+|b|$.
		\item $a+b\geq c\iff a\geq c-b$.
	\end{enumerate}
\end{lemma}
\begin{proof}
	These follow immediately for finite values and can be easily checked for infinite values. 
\end{proof}

\begin{lemma}\label{prop:lim_of_sum}
  Given two sequences $(a_n)$ and $(b_n)$ in $\mathbb{R}_+$,
  \[
    \lim_{n\to\infty}(a_n+b_n)=\lim_{n\to\infty}a_n+\lim_{n\to\infty}b_n,
  \]
  provided both limits on the right-hand side exist.
\end{lemma}
\begin{proof}
  Proving the statement when the limits of $(a_n)$ and $(b_n)$ are
  finite is standard. If, without loss of generality, $a_n\to\infty$, then by definition the sequence
  $(a_n)$ is eventually constant on $\infty$, so that $(a_n+b_n)$ is eventually
  infinite as well, and the result follows.
\end{proof}

\begin{lemma}\label{prop:lim_of_comparables}
  Let $(a_n)$ and $(b_n)$ be two convergent sequences in $\mathbb{R}_+$, and
  suppose there exists some natural number $N$ such that $a_n\leq b_n$ for all
  $n\geq N$. Then
  \[
    \lim_{n\to\infty}a_n\leq\lim_{n\to\infty}b_n.
  \]
  In particular, if $(a_n)$ is eventually bounded above by a constant
  $M\in\mathbb{R}_+$, then
  \[
    \lim_{n\to\infty}a_n\leq M.
  \]
\end{lemma}


\begin{lemma}\label{lemma:equivalents-and-cauchy-coincide}
  Let $(x_n)$ and $(y_n)$ be sequences in an $\RR_+$-category $\CC$. Then  
  \begin{enumerate}[label={(\roman*)},noitemsep]
    \item If $(x_n)$ is Cauchy, every subsequence of $(x_n)$ is equivalent to
    $(x_n)$.
    \item If $(x_n)$ is equivalent to $(y_n)$, $(x_n)$ is Cauchy if and only if
    $(y_n)$ is. 
    \item If $(x_n)$ is equivalent to $(y_n)$, $z\in\mathbb{C}$ is a limit of
    $(x_n)$ if and only if $z$ is a limit of $(y_n)$.
	\end{enumerate}
\end{lemma}

\begin{lemma}\label{lemma:diagonal-is-equivalent}
  Let $(x_n)$ be a sequence in an $\bR_+$-category $\bC$ such that each $x_n$
  is a limit of some sequence $(x_{n,m})_{m\in\NN}$. Then there exists a
  ``diagonal sequence'' $(d_n)$ where each $d_n$ belongs to the sequence
  $(x_{n,m})_{m\in\mathbb{N}}$, and $(d_n)$ is equivalent to $(x_n)$.
\end{lemma}
\begin{proof}[Construction]
  For each $n\in\mathbb{N}$, set $d_n:=x_{n,N_n}$, where $N_n\in\mathbb{N}$
  satisfies
  \[
    m\geq N_n\implies\max(\bC(x_{n,m},x_n),\bC(x_n,x_{n,m}))<\frac1n.\qedhere
  \]
\end{proof}

\begin{remark}\label{rmk:equivalence_of_Cauchy_gives_equiv_relation}
Equivalence of Cauchy sequences as in \autoref{def:equivalence_of_Cauchy_seqs} is an equivalence relation on the set of Cauchy sequences in a symmetric $\bR_+$-category $\bC$. The transitivity of this relation follows from \autoref{prop:lim_of_sum}.
\end{remark}

\begin{lemma}\label{lemma:lim_of_cauchy_seqs_exists}
  Given a symmetric ${\bR_+}$-enriched category $\bC$ and two Cauchy sequences
  $(x_n)$ and $(y_n)$ in $\bC$, the following limit exists in $\mathbb{R}_+$.
  \[
    \lim_{n\to\infty}\bC(x_n,y_n)
  \]
\end{lemma}
\begin{proof}
  Note $(\bC(x_n,y_n))_{n\in\bN}$ is a sequence in ${\bR_+}$, so in order to
  show it has a unique limit it suffices to show it is Cauchy.%
  \footnote{
    This is because $\bR_+$ is a complete extended metric space, i.e., every
    Cauchy sequence has a limit in $\bR_+$.
  } 
  Let $\vare>0$ be given. Choose $N$ such that $\CC(x_n,x_m)$ and
  $\CC(y_n,y_m)$ are both bounded above by $\frac{\vare}{2}$ for $n,m\geq
  N$. Let $n,m\geq N$, and suppose without loss of generality that
  $\CC(x_n,y_n)\geq \CC(x_m,y_m)$. Then 
	\begin{align*}
		|\bC(x_n,y_n)-\bC(x_m,y_m)|
    &\leq(\bC(x_n,x_m)+\bC(x_m,y_n))-\bC(x_m,y_m) \\
		&\leq(\bC(x_n,x_m)+(\bC(x_m,y_m)+\bC(y_m,y_n)))-\bC(x_m,y_m) \\
		&=((\bC(x_n,x_m)+\bC(y_m,y_n))+\bC(x_m,y_m))-\bC(x_m,y_m) \\
    &\leq\mathbb{C}(x_n,x_m)+\mathbb{C}(y_m,y_n) \\
    &<\frac{\varepsilon}{2}+\frac{\varepsilon}{2}=\varepsilon,
	\end{align*}
  where the second-to-last inequality is
  \autoref{lemma:arithmetic_in_R+}(i). Hence $\CC(x_n,y_n)$ is a Cauchy
  sequence in $\mathbb{R}_+$, as desired.
\end{proof}

\begin{lemma}\label{prop:isomorphic_limits}
  Let $(x_n)$ be a sequence with a limit $y$ in an $\bR_+$-category $\bC$. Then
  for each $z\in\bC$, the following are equivalent:
  \begin{enumerate}[label=(\arabic*)]
		\item The objects $y$ and $z$ are isomorphic (i.e., $\bC(y,z)=\bC(z,y)=0$)
		\item The object $z$ is also a limit of $\{x_n\}$.
	\end{enumerate}
	In particular, if $\CC$ is a gaunt $\bR_+$-category, then limits in $\CC$ are unique.
\end{lemma}
\begin{proof}
  Let $y$ be a limit of $(x_n)$, then if $z\cong y$, we have by the triangle
  inequality and symmetry that $\mathbb{C}(x_n,z)=\mathbb{C}(x_n,y)$ and
  $\mathbb{C}(z,x_n)=\mathbb{C}(y,x_n)$, so that since
  $\mathbb{C}(x_n,y),\mathbb{C}(y,x_n)\to 0$, so it follows that
  $\mathbb{C}(x_n,z),\mathbb{C}(z,x_n)\to 0$, meaning $z$ is a limit of
  $(x_n)$, as desired. On the other hand, if $y$ and $z$ are limits of $(x_n)$,
  given a positive integer $n$, by the triangle inequality,
  \autoref{prop:lim_of_comparables}, and \autoref{prop:lim_of_sum}, we have
  \[
    \mathbb{C}(y,z)
    =\lim_{n\to\infty}\mathbb{C}(y,z)
    \leq\lim_{n\to\infty}(\mathbb{C}(x_n,z)+\mathbb{C}(y,x_n))
    =0,
  \]
  where the limits are guaranteed to exist by
  \autoref{lemma:lim_of_cauchy_seqs_exists}. By symmetry we have
  $\mathbb{C}(z,y)=0$ as well.
\end{proof}

\begin{lemma}\label{lemma:rev_tri_ineq_limit_version}
  Given a symmetric $\bR_+$-category $\bC$ and three Cauchy sequences $(x_n)$,
  $(y_n)$, and $(z_n)$ in $\bC$,
  \begin{align*}
    \lim_{n\to\infty}\bC(x_n,z_n)-\lim_{n\to\infty}\bC(y_n,z_n)
    &\leq\left|\lim_{n\to\infty}\bC(x_n,z_n)-\lim_{n\to\infty}\bC(y_n,z_n)\right| \\
    &\leq\lim_{n\to\infty}\bC(x_n,y_n).
  \end{align*}
\end{lemma}
\begin{proof}
  The existence of the limits involved is
  \autoref{lemma:lim_of_cauchy_seqs_exists}. The first inequality is clear by
  symmetry. Now, suppose without loss of generality that
  \[
    \lim_{n\to\infty}\mathbb{C}(y_n,z_n)
    \leq \lim_{n\to\infty}\mathbb{C}(x_n,z_n),
  \]
  Then by the triangle inequality,
  \autoref{prop:lim_of_comparables}, and \autoref{prop:lim_of_sum}, we have
  \[
    \lim_{n\to\infty}\bC(x_n,z_n)
    \leq\lim_{n\to\infty}(\bC(x_n,y_n)+\bC(y_n,z_n))
    =\lim_{n\to\infty}\bC(x_n,y_n)+\lim_{n\to\infty}\bC(y_n,z_n).
  \]
	Therefore,
	\begin{align*}
    \lim_{n\to\infty}\mathbb{C}(x_n,z_n)-&\lim_{n\to\infty}\mathbb{C}(y_n,z_n) \\
    &\leq\left(\lim_{n\to\infty}\bC(x_n,y_n)+\lim_{n\to\infty}\bC(y_n,z_n)\right)-\lim_{n\to\infty}\bC(y_n,z_n) \\
		&\leq\lim_{n\to\infty}\bC(x_n,y_n),
	\end{align*}
	where the last inequality is \autoref{lemma:arithmetic_in_R+}(i). The desired result follows.
\end{proof}

\begin{lemma}\label{lemma:weird_triangle_inequality_limit_lemma_thingy}
  Given an ${\bR_+}$-category $\bC$ and three Cauchy sequences $(x_n)$,
  $(y_n)$, and $(z_n)$ in $\bC$ such that
  \[
    \lim_{n\to\infty}\bC(x_n,y_n)
    <\infty,
  \]
  we have
  \[
    \lim_{n\to\infty}\bC(x_n,z_n)=\infty
    \ \iff\ 
    \lim_{n\to\infty}\bC(y_n,z_n)=\infty.
  \]
\end{lemma}
\begin{proof}
  By symmetry, it suffices to show one direction. Suppose that 
  \[
    \lim_{n\to\infty}\CC(x_n,z_n)<\infty.
  \]
  Then by the triangle inequality, \autoref{prop:lim_of_comparables}, and
  \autoref{prop:lim_of_sum}, we have
  \[
    \lim_{n\to\infty}\CC(y_n,z_n)
    \leq\lim_{n\to\infty}(\mathbb{C}(x_n,z_n)+\mathbb{C}(y_n,x_n))
    <\infty.\qedhere
  \]
\end{proof}

\begin{lemma}\label{lemma:double_indexed_limit_can_be_exchanged_for_one_limit}
  Given a symmetric ${\bR_+}$-enriched category $\bC$ and two Cauchy sequences
  $(x_n)$ and $(y_n)$ in $\bC$,
  \[
    \lim_{m\to\infty}\lim_{n\to\infty}\bC(x_m,y_n)
    =\lim_{n\to\infty}\bC(x_n,y_n).
  \]
\end{lemma}
\begin{proof}
  By \autoref{lemma:lim_of_cauchy_seqs_exists}, we know that
  \[
    D:=\lim_{n\to\infty}\mathbb{C}(x_n,y_n)
    \qquad\text{and}\qquad
    L_m:=\lim_{n\to\infty}\bC(x_m,y_n)
  \]
  exist in $\mathbb{R}_+$ for each $m\in\mathbb{N}$. Then the statement we wish
  to show is that $L_m\to D$. Let $\vare>0$ be given. Since $(x_n)$ is a Cauchy
  sequence, there exists some $M\in\bN$ such that
  $\mathbb{C}(x_n,x_m)<\frac{\varepsilon}{2}$ for all $n,m\geq M$. Fix $m\geq
  M$. We know
  \[
    \lim_{n\to\infty}\bC(x_n,x_m)\leq\frac\vare2,
  \]
  by Lemmas \ref{lemma:lim_of_cauchy_seqs_exists} and
  \ref{prop:lim_of_comparables}. Thus, by
  \autoref{lemma:weird_triangle_inequality_limit_lemma_thingy} applied to the
  sequences $(x_n)$, $\widehat{x_m}$,%
  \footnote{
    Recall $\widehat{x_m}$ denotes the constant sequence on $x_m$.
  }
  and $(y_n)$, we have $D=\infty$ if and only if $L_m=\infty$. Hence, we may
  split the remaining proof into two cases.
	
  \textbf{Case 1:} $L_m=D=\infty$. In this case, we've shown that for all
  $m\geq M$,
	\[\left|L_m-D\right|=|\infty-\infty|=|0|=0<\vare.\]
	
  \textbf{Case 2:} $D, L_m<\infty$. In this case we by
  \autoref{lemma:rev_tri_ineq_limit_version} that
	\begin{equation*}
		\left|L_m-D\right|
    =\left|\lim_{n\to\infty}\bC(x_m,y_n)-\lim_{n\to\infty}\bC(x_n,y_n)\right|
    \leq\lim_{n\to\infty}\bC(x_m,x_n)
    \leq\frac\vare2
    <\vare.
	\end{equation*}
  In both cases, we have shown that $|L_m-D|<\varepsilon$ for all $m\geq M$, as
  desired.
\end{proof}

\begin{lemma}\label{prop:R-functors-preserve-everything}
	Given an $\bR_+$-functor $f:\bC\to\bD$,
	\begin{enumerate}[label=(\roman*)]
		\item $f$ sends Cauchy sequences to Cauchy sequences.
    \item $f$ sends a limit of a sequence $(x_n)_{n\in\mathbb{N}}$ to a limit
    of $(f(x_n))_{n\in\mathbb{N}}$.
    \item $f$ sends equivalent sequences in $\mathbb{C}$ to equivalent
    sequences in $\mathbb{D}$.
	\end{enumerate}
  Furthermore, if $f$ is fully faithful, it reflects Cauchy sequences, limits,
  and equivalent sequences in its image.
\end{lemma}

\begin{lemma}\label{prop:yoneda_fully_faithful}
  Given an $\bR_+$-category $\bC$, the Yoneda embedding
  $\cY_\bC:\bC\to{(\mathbb{R}_+)}_\bC$ as defined in
  \autoref{def:R-yoneda-embedding} is a fully faithful $\bR_+$-functor (in the
  sense of \autoref{def:collected_R+defs}).
\end{lemma}
\begin{proof}
  First, it is straightforward to check that for each object $c$ in $\bC$, the
  assignment $c'\mapsto\bC(c',c)$ defines a valid $\bR_+$-enriched presheaf
  $\bC^\op\to\bR_+$. Let $x, y\in\CC$. By unraveling definitions, we get
  that
  \begin{align*}
    (\mathbb{R}_+)_{\CC}(\cY(x),\cY(y))
    &=\sup_{z\in\bC}\max(\CC(z,y)-\CC(z,x),0) \\
    &\geq\max(\mathbb{C}(x,y)-\mathbb{C}(x,x),0) \\
    &=\bC(x,y).
  \end{align*}
  
  Proceeding to the other inequality, we begin with the triangle inequality and apply \autoref{lemma:arithmetic_in_R+}(iii)
\begin{align*}
    \bC(z,y) &\leq \bC(z,x) +\bC(x,y)  & \\
    \bC(z,y) - \bC(z,x) &\leq \bC(x,y) &\text{(\autoref{lemma:arithmetic_in_R+}(iii))}\\
\end{align*}

Applying supremum and maximum produces the following inequality: 

  \[
    (\mathbb{R}_+)_{\CC}(\cY(x),\cY(y))
    =\sup_{z\in\bC}\max(\CC(z,y)-\CC(z,x),0)\leq \sup_{z\in\CC}\max(\CC(x, y), 0)
    =\CC(x,y).\qedhere
  \]
\end{proof}

\begin{lemma}\label{prop:ell_x_is_short_map}
  Given a symmetric $\bR_+$-category $\bC$ and a Cauchy sequence $(x_n)$ in
  $\bC$, the function $\ell_{(x_n)}$  of
  \autoref{def:presheaf_associated_to_cauchy_sequence} is an $\bR_+$-functor
  from $\bC^\op$ to $\bR_+$.
\end{lemma}
\begin{proof}
  Let $y, z\in\CC$, then
	\begin{align*}
    \bR_+(\ell_{(x_n)}(y),\ell_{(x_n)}(z))
    &=\max\left(\lim_{n\to\infty}\bC(z,x_n)-\lim_{n\to\infty}\bC(y,x_n),0\right) \\
    &\leq\max\left(\lim_{n\to\infty}\bC(z,y),0\right) &\text{(\autoref{lemma:rev_tri_ineq_limit_version})} \\
		&=\bC(z,y)=\bC(y,z).\qedhere
	\end{align*}
\end{proof}

\begin{proposition}\label{prop:equivalent_Cauchy_characterization}
	Let $\bC$ be a symmetric $\RR_+$-category and $f:\CC^\op\to \RR_+$ an $\RR_+$-enriched presheaf. The following are equivalent. 
	\begin{enumerate}
		\item The presheaf $f$ is a limit (in $(\RR_+)_\CC$) of a Cauchy sequence of objects in the image of the Yoneda embedding. 
		\item The presheaf $f$ has a dual, i.e., $f$ belongs to the Cauchy completion $\ol\bC$ of $\bC$ (\autoref{def:cauchy_completion}).
    \item There is a Cauchy sequence $(x_n)$ in $\bC$ such that
    $f=\ell_{(x_n)}$ (where $\ell_{(x_n)}$ is as defined in
    \autoref{def:presheaf_associated_to_cauchy_sequence}).
	\end{enumerate}
    This is \autoref{prop:equivalent_Cauchy_characterization_fake}
\end{proposition}
\begin{proof}
  We first note that the argument of \cite[Example 3]{BorceuxDejean} extends,
  \emph{mutatis mutandis}, to show the equivalence of (2) and (3). We now show
  that $(1)\implies(3)$. Let $f$ be a limit of some Cauchy sequence
  $(\cY_\bC(x_n))$ in $(\mathbb{R}_+)_\bC$, where $(x_n)$ is a sequence in
  $\bC$. In particular, the Yoneda embedding $\cY_\bC$ is a fully faithful
  $\bR_+$-functor by \autoref{prop:yoneda_fully_faithful}, so by
  \autoref{prop:R-functors-preserve-everything}, $(x_n)$ is a Cauchy sequence
  in $\bC$. We claim that $f=\ell_{(x_n)}$. Let $y\in\bC$. Then we want to show
  \[
    \ell_{(x_n)}(y)
    =\lim_{n\to\infty}\bC(y,x_n)
    =f(y).
  \]
  Let $\vare>0$. Because $f$ is a limit of the sequence $(\cY_\bC(x_n))$, there
  exists $N\in\bN$ such that
  \[
    n\geq N
    \implies
    \max\left((\mathbb{R}_+)_\bC(f,\cY_\bC(x_n)),{(\bR_+)}_\bC(\cY_\bC(x_n),f)\right)<\varepsilon.
  \]
	Then, for every fixed $n\in\bN$, we have
	\begin{align*}
    |f(y)-\mathbb{C}(y,x_n)|
    &=\max\left(\bR_+(f(y),\mathbb{C}(y,x_n)),\bR_+(\mathbb{C}(y,x_n),f(y))\right) \\
    &\leq\max\left(\sup_{z\in\mathbb{C}}\bR_+(f(z),\mathbb{C}(z,x_n)),\sup_{z\in\mathbb{C}}\bR_+(\mathbb{C}(z,x_n),f(z))\right) \\
    &=\max\left({(\mathbb{R}_+)}_\mathbb{C}(f,\cY_\mathbb{C}(x_n)),{(\bR_+)}_\mathbb{C}(\cY_\mathbb{C}(x_n),f)\right)<\vare.
	\end{align*}
  Hence the sequence $(\mathbb{C}(y,x_n))$ limits to $f(y)$, so that
  $\ell_{(x_n)}(y)=f(y)$, because limits in $\mathbb{R}_+$ are unique. Thus,
  $f=\ell_{(x_n)}$, so we've shown (1)$\implies$(3).
	
  We now wish to show that (3)$\implies$(1). Let $\ell_{(x_n)}$ be the presheaf
  determined by some Cauchy sequence $(x_n)$ in $\CC$. Then by
  \autoref{prop:R-functors-preserve-everything}, since $(x_n)$ is a Cauchy
  sequence, so is $(\cY_\bC(x_n))$. We claim $\ell_{(x_n)}$ is a limit of
  this sequence. Let $\vare>0$. We want to find some $N\in\bN$ such that
	\begin{equation}\label{eq:1}
    \max\left({(\mathbb{R}_+)}_\bC(\cY_\bC(x_n),\ell_{(x_n)}),{(\bR_+)}_\bC(\ell_{(x_n)},\cY_\bC(x_n))\right)<\vare
    \quad\text{for all }n\geq N.
	\end{equation}
	Since $(x_n)$ is a Cauchy sequence, there exists $N\in\bN$ such that for all $n\geq N$, we have
	\[\bC(x_n,x_m)<\frac\vare2.\]
	Then given $n\geq N$,
	\begin{align*}
    &\max({(\mathbb{R}_+)}_\bC(\cY_\bC(x_n),\ell_{(x_n)}),{(\bR_+)}_\bC(\ell_{(x_n)},\cY_\bC(x_n))) \\
		&=\max\left(\sup_{z\in\bC}\bR_+\left(\bC(z,x_n),\lim_{m\to\infty}\bC(z,x_m)\right),\sup_{z\in\bC}\bR_+\left(\lim_{m\to\infty}\bC(z,x_m),\bC(z,x_n)\right)\right) \\
		&\leq\sup_{z\in\bC}\left|\bC(z,x_n)-\lim_{m\to\infty}\bC(z,x_m)\right| \\
		&=\sup_{z\in\bC}\left|\lim_{m\to\infty}\bC(z,x_n)-\lim_{m\to\infty}\bC(z,x_m)\right| \\
		&\overset{(\ast)}\leq\sup_{z\in\bC}\lim_{m\to\infty}\bC(x_n,x_m) \\
		&=\lim_{m\to\infty}\bC(x_n,x_m),
	\end{align*}
  where $(\ast)$ follows by
  \autoref{lemma:rev_tri_ineq_limit_version}. Since $\bC(x_n,x_m)<\vare$
  for all $m\geq N$ (because $n\geq N$), we have by
  \autoref{prop:lim_of_comparables} that
  \[
    \lim_{m\to\infty}\bC(x_n,x_m)\leq\frac\vare2<\vare.
  \]
  Hence, the desired inequality (\autoref{eq:1}) has been shown, $\ell_{(x_n)}$
  is the limit of the sequence $(\cY_\bC(x_n))$.
\end{proof}

As a consequence of the proof of
\autoref{prop:equivalent_Cauchy_characterization}, we have the following
Corollary.

\begin{corollary}\label{cor:lim_iota_cauchy_seq}
  Given a Cauchy sequence $(x_n)$ in $\bC$, the presheaf $\ell_{(x_n)}$
  (\autoref{def:presheaf_associated_to_cauchy_sequence}) is a limit of the
  sequence $(\iota_\mathbb{C}(x_n))_{n\in\mathbb{N}}$ in
  $\overline{\mathbb{C}}$, where $\iota_{\mathbb{C}}:\mathbb{C}\to
  \overline{\mathbb{C}}\subseteq{(\mathbb{R}_+)}_{\mathbb{C}}$ is the Cauchy
  completion $\mathbb{R}_+$-functor (\autoref{remark:iota_C_Cauchy_completion})
  obtained by restricting the codomain of the $\mathbb{R}_+$-Yoneda embedding
  $\mathcal{Y}_\mathbb{C}:\mathbb{C}\to{(\mathbb{R}_+)}_{\mathbb{C}}$.
\end{corollary}

\begin{lemma}\label{prop:equiv-cauchy-same-presheaf} 
	Let $(x_n)$ and $(y_n)$ be Cauchy sequences in a symmetric $\RR_+$-category $\CC$. The following are equivalent: 
	\begin{enumerate}
		\item The Cauchy sequences $(x_n)$ and $(y_n)$ are equivalent, i.e., $(x_n)\sim (y_n)$.
		\item The preshezves $\ell_{(x_n)}$ and $\ell_{(y_n)}$ are \emph{equal}.
	\end{enumerate}
\end{lemma}
\begin{proof}
	First, assume (1).
  Let $z\in\CC$ be arbitrary. It suffices to show that $\ell_{(x_n)}(z) =
  \ell_{(y_n)}(z)$. To that end, an application of the triangle inequality
  gives us the following:
	\begin{align*}
		\ell_{(x_n)}(z)=\lim_{n\to\infty} \CC(z, x_n) &\leq \lim_{n\to\infty} \left(\CC(z, y_n) + \CC(y_n, x_n)\right) &&(\text{\autoref{prop:lim_of_comparables}})\\
		&= \lim_{n\to\infty}\CC(z, y_n) + \lim_{n\to\infty}\CC(y_n, x_n) &&(\text{\autoref{prop:lim_of_sum}})\\
		&= \lim_{n\to\infty}\CC(z, y_n)=\ell_{(y_n)}(z) &&(\text{since $(x_n)\sim(y_n)$}).
	\end{align*}
	Repeating this exact same argument but with the role of $x_n$ and $y_n$ swapped yields the reverse inequality.
	Now assume (2).  Then we have
	\begin{align*}
		\lim_{n\to\infty}\CC(x_n, y_n) &= \lim_{m\to\infty}\lim_{n\to\infty}\CC(x_m, y_n) &&\text{(\autoref{lemma:double_indexed_limit_can_be_exchanged_for_one_limit})}\\
		&= \lim_{m\to\infty}\lim_{n\to\infty}\CC(x_m, x_n) &&\text{(since $\ell_{(x_n)}(x_m) = \ell_{(y_n)}(x_m)$)}\\
		&= \lim_{n\to\infty}\CC(x_n, x_n)=0 &&\text{(\autoref{lemma:double_indexed_limit_can_be_exchanged_for_one_limit})},
	\end{align*}
	so that indeed $(x_n)\sim(y_n)$.
\end{proof}

\begin{lemma}\label{prop:Lmetric-on-cauchy-completion}
  Given two Cauchy sequences $(x_n)$ and $(y_n)$ in a symmetric
  ${\bR_+}$-category $\bC$, we have
  \[
    \overline\bC(\ell_{(x_n)},\ell_{(y_n)})=\lim_{n\to\infty}\CC(x_n,y_n),
  \]
  where $\overline{\mathbb{C}}\subseteq{(\mathbb{R}_+)}_{\mathbb{C}}$ denotes
  the Cauchy completion of $\mathbb{C}$, as defined in
  \autoref{def:cauchy_completion}.
\end{lemma}
\begin{proof}
  Unraveling definitions and applying the limit version of the reverse
  triangle inequality (\autoref{lemma:rev_tri_ineq_limit_version}) and symmetry
  yields that
  \begin{align*}
    \overline{\mathbb{C}}(\ell_{(x_n)},\ell_{(y_n)})
    &=\sup_{z\in\mathbb{C}}\max\left(\lim_{n\to\infty}\mathbb{C}(z,y_n)-\lim_{n\to\infty}\mathbb{C}(z,x_n),0\right) \\
    &\leq\sup_{z\in\mathbb{C}}\max\left(\lim_{n\to\infty}\mathbb{C}(y_n,x_n),0\right) \\
    &=\lim_{n\to\infty}\mathbb{C}(y_n,x_n)
    =\lim_{n\to\infty}\mathbb{C}(x_n,y_n).
  \end{align*}
  To show the opposite inequality holds, first note
	\begin{align*}
		\overline\bC(\ell_{(x_n)},\ell_{(y_n)})&=\sup_{z\in\bC}\max\left(\lim_{n\to\infty}\bC(z,y_n)-\lim_{n\to\infty}\bC(z,x_n),0\right)\\ 
		&\geq\lim_{m\to\infty}\max\left(\lim_{n\to\infty}\bC(x_m,y_n)-\lim_{n\to\infty}\bC(x_m,x_n),0\right).
	\end{align*}
  Now, define $D:=\lim_{n\to\infty}\mathbb{C}(x_n,y_n)$. We claim
  \[
    \lim_{m\to\infty}\max\left(\lim_{n\to\infty}\bC(x_m,y_n)-\lim_{n\to\infty}\bC(x_m,x_n),0\right)=D,
  \]
  that will yield the desired result. Let $\varepsilon>0$. By
  \autoref{lemma:double_indexed_limit_can_be_exchanged_for_one_limit}, we know
  that
  \[
    D=\lim_{m\to\infty}\lim_{n\to\infty}\bC(x_m,y_n)
    \quad\text{and}\quad
    \lim_{m\to\infty}\lim_{n\to\infty}\bC(x_m,x_n)
    =\lim_{n\to\infty}\bC(x_n,x_n)=0,
  \]
  so we can choose some $M\in\mathbb{N}$ such that for all $m\geq M$,
  \[
    \left|\lim_{n\to\infty} \CC(x_m,y_n)-D\right|<\frac{\varepsilon}{2} 
    \qquad\text{and}\qquad
    \lim_{n\to\infty} \CC(x_m,x_n)<\frac{\varepsilon}{2} .
  \]
	Then for $m\geq M$, we have
  \begin{multline*}
    \left|\left(\lim_{n\to\infty} \CC(x_m,y_n)-\lim_{n\to \infty} \CC(x_m,x_n)\right)-D\right| \\
    \overset{(\ast)}\leq \left|\lim_{n\to\infty} \CC(x_m,y_n)-D\right|+\left|\CC(x_m,x_n)\right| < \varepsilon,
  \end{multline*}
  where $(\ast)$ follows by \autoref{lemma:arithmetic_in_R+}(ii). Thus we've shown
  \[
    \lim_{m\to\infty}\left(\lim_{n\to\infty}\mathbb{C}(x_m,y_n)-\lim_{n\to\infty}\mathbb{C}(x_m,x_n)\right)=D.
  \]
  Because $D=\max(D,0)$, it suffices to check $x\mapsto\max(x,0)$ is continuous
  as a function $[-\infty,\infty]\to[-\infty,\infty]$ (i.e., that it commutes
  with taking limits of sequences $(x_n)$ in $[-\infty,\infty]$), and this is
  entirely standard.
  %
	%
  %
  %
\end{proof}

\begin{lemma}\label{lemma:limits_can_be_moved_out_of_distance_function}
  Given a symmetric $\bR_+$-category $\bC$ and two Cauchy sequences $(x_n)$ and
  $(y_n)$ in $\mathbb{C}$ with limits $x$ and $y$ respectively, we have
	\[\lim_{n\to\infty}\bC(x_n,y_n)=\bC(x,y).\]
\end{lemma}
\begin{proof}
  We have
  \[
    \lim_{n\to\infty}\mathbb{C}(x_n,y_n)
    =\overline{\mathbb{C}}(\ell_{(x_n)},\ell_{(y_n)})
    =\overline{\mathbb{C}}(\ell_{\widehat{x}},\ell_{\widehat{y}})
    =\overline{\mathbb{C}}(\mathcal{Y}_\mathbb{C}(x),\mathcal{Y}_\mathbb{C}(y))
    =\mathbb{C}(x,y),
  \]
  where the first equality is \autoref{prop:Lmetric-on-cauchy-completion}, the
  second equality is \autoref{prop:equiv-cauchy-same-presheaf} (since clearly
  $x_n\to x$ if and only if $(x_n)$ and $\widehat{x}$ are equivalent Cauchy
  sequences, by definition), the third equality follows by unraveling
  definitions, and the last equality is because the $\mathbb{R}_+$-Yoneda
  embedding is fully faithful (\autoref{prop:yoneda_fully_faithful}).
\end{proof}

\begin{lemma}\label{prop:R-cauchy-complete-is-cauchy-complete}
  Every Cauchy sequence in $\ol\CC$ has a limit.
\end{lemma}
\begin{proof}
  Suppose $(\ell_n)_{n\in\mathbb{N}}$ is a Cauchy sequence in
  $\overline{\mathbb{C}}$, so by
  \autoref{prop:equivalent_Cauchy_characterization}, for each $n\in\mathbb{N}$
  there exists a Cauchy sequence $\overline{x_n}=(x_{n,m})_{m\in\mathbb{N}}$ in
  $\mathbb{C}$ such that $\ell_n=\ell_{\overline{x_n}}$. By
  \autoref{cor:lim_iota_cauchy_seq}, each $\ell_{\ol{x_n}}$ is a limit of the
  sequence $(\iota_\bC(x_{n,m}))_{m\in\mathbb{N}}$ in $\ol\bC$. By
  \autoref{lemma:diagonal-is-equivalent}, we can therefore pick a diagonal
  sequence $(\iota_{\mathbb{C}}(d_n))_{n\in\mathbb{N}}$ in the image of
  $\iota_\bC$ satisfying $(\iota_{\mathbb{C}}(d_n))\sim(\ell_{\ol{x_n}})$. The
  sequence $(\iota_{\mathbb{C}}(d_n))$ is Cauchy by
  \autoref{lemma:equivalents-and-cauchy-coincide}(ii), and $\iota_\mathbb{C}$
  is fully faithful, so $(d_n)$ is Cauchy in $\mathbb{C}$ again by
  \autoref{lemma:equivalents-and-cauchy-coincide}. Thus by
  \autoref{cor:lim_iota_cauchy_seq} we have that $\ell_{(d_n)}$ is a limit of
  $(\iota_{\mathbb{C}}(d_n))$ in $\overline{\mathbb{C}}$. Finally, since
  $(\ell_n)$ is equivalent to $(\iota_\mathbb{C}(d_n))$, it follows by
  \autoref{lemma:equivalents-and-cauchy-coincide}(iii) that $\ell_{(d_n)}$ is a
  limit of $(\ell_n)$, as desired.
\end{proof}

\begin{lemma}\label{cor:R-Cauchy_completion_is_gaunt}
  The Cauchy completion of $\overline{\mathbb{C}}$ of any symmetric
  $\RR_+$-category $\mathbb{C}$ is gaunt, and thus has unique limits when they
  exist (\autoref{prop:isomorphic_limits}).
\end{lemma}
\begin{proof}
  Let $\ell_1,\ell_2\in \overline{\mathbb{C}}$ such that
  \[
    \overline{\mathbb{C}}(\ell_1,\ell_2)=0=\overline{\mathbb{C}}(\ell_2,\ell_1).
  \]
  By \autoref{prop:equivalent_Cauchy_characterization} there exists Cauchy
  sequences $(x_n)$ and $(y_n)$ in $\mathbb{C}$ such that $\ell_1=\ell_{(x_n)}$
  and $\ell_2=\ell_{(y_n)}$. Then the above equalities together with
  \autoref{prop:Lmetric-on-cauchy-completion} and
  \autoref{prop:equiv-cauchy-same-presheaf} yield the desired result.
\end{proof}

\begin{lemma}
  The following are equivalent for a symmetric $\mathbb{R}_+$-category
  $\mathbb{C}$.
  \begin{enumerate}[label=(\roman*)]
    \item $\iota_\bC:\bC\to\overline\bC$ is an ${\bR_+}$-equivalence of
    categories.
		\item Every Cauchy sequence in $\bC$ has a limit.
	\end{enumerate}
  Furthermore, if either of the above equivalent conditions hold, we say that
  $\bC$ is \emph{Cauchy complete}.
\end{lemma}
\begin{proof}
  First, we show (1)$\implies$(2). Suppose $\iota_\bC$ is an equivalence of
  categories, and let $(x_n)$ be a Cauchy sequence in $\bC$. By
  \autoref{prop:R-cauchy-complete-is-cauchy-complete}, $(\iota_\bC(x_n))$ is a
  Cauchy sequence in $\ol\bC$, so that it has a limit $\ell$ in $\ol\bC$. Then
  since $\iota_\bC$ is an equivalence, it is essentially surjective, so
  there exists $y\in\bC$ such that $\iota_\bC(y)\cong \ell$. Then by
  \autoref{prop:isomorphic_limits}, $\iota_\bC(y)$ is likewise a limit of
  $(\iota_\bC(x_n))$. Finally, by
  \autoref{prop:R-functors-preserve-everything}, since $\iota_{\mathbb{C}}$ is
  fully faithful, it reflects limits in its image, so that $y$ is a limit of
  $(x_n)$, as desired.
	
  Next we claim that (2)$\implies$(1). Suppose that every Cauchy sequence in
  $\bC$ has a limit. Then we want to show that $\iota_\bC$ is surjective, as
  $\iota_\bC$ is always a fully faithful $\bR_+$-functor. Let $\ell_{(x_n)}$ be
  a presheaf in $\ol\bC$ determined by a Cauchy sequence $(x_n)$ in $\bC$
  (\autoref{def:presheaf_associated_to_cauchy_sequence}). Then $(x_n)$
  has a limit $y$, so that $\iota_\bC(y)$ is a limit of $(\iota_\bC(x_n))$ by
  \autoref{prop:R-functors-preserve-everything}. Yet, $\ell_{(x_n)}$ is also a
  limit of $(\iota_\bC(x_n))$ (\autoref{cor:lim_iota_cauchy_seq}) and limits in
  $\ol\bC$ are unique by \autoref{cor:R-Cauchy_completion_is_gaunt}, so that
  necessarily $\iota_\bC(y)=\ell_{(x_n)}$. Therefore, $\iota_\bC$ is
  \textit{strictly} surjective and fully faithful, and thus an
  equivalence.
\end{proof}

\begin{lemma}\label{prop:induced_R-functor_between_completions}
	Given an ${\bR_+}$-functor $f:\bC\to\bD$ between symmetric $\bR_+$-categories, there exists a \emph{unique} induced ${\bR_+}$-functor $\overline f:\overline\bC\to\overline\bD$ such that the following diagram commutes.
	\[\begin{tikzcd}
		\bC & \bD \\
		\overline\bC & \overline\bD
		\arrow["f", from=1-1, to=1-2]
		\arrow["{\iota_\bD}", from=1-2, to=2-2]
		\arrow["{\iota_\bC}"', from=1-1, to=2-1]
		\arrow["{\overline f}", from=2-1, to=2-2]
	\end{tikzcd}\]
\end{lemma}
\begin{proof}
  By \autoref{prop:equivalent_Cauchy_characterization}, every object in
  $\overline{\mathbb{C}}$ is of the form $\ell_{(x_n)}$ for some Cauchy
  sequence $(x_n)$ in $\mathbb{C}$. Now, define
  $\overline{f}:\overline{\mathbb{C}}\to\overline{\mathbb{D}}$ by
  $\overline{f}(\ell_{(x_n)})=\ell_{(f(x_n))}$. This assignment is well-defined
  and clearly makes the diagram commute by Lemmas
  \ref{prop:R-functors-preserve-everything} and
  \ref{prop:equiv-cauchy-same-presheaf}. It is an $\mathbb{R}_+$-functor by
  Lemmas \ref{prop:Lmetric-on-cauchy-completion} and
  \ref{prop:lim_of_comparables}. It only remains to show that $\ol f$ is
  unique. Suppose $\widetilde{f}:\ol\bC\to\ol\bD$ is an $\mathbb{R}_+$-functor
  satisfying $\widetilde{f}\circ\iota_\bC=\iota_\bD\circ f$. By
  \autoref{cor:R-Cauchy_completion_is_gaunt}, since  $\overline{\mathbb{C}}$
  is gaunt, to show that $\ol
  f=\widetilde{f}$,  it suffices to show that
  \[
    \ol\bD(\ol f(\ell_{(x_n)}),\widetilde{f}(\ell_{(x_n)}))=0
  \]
  for all Cauchy sequences $(x_n)$ in $\mathbb{C}$. Let $(x_n)$ be a Cauchy sequence in $\mathbb{C}$. Then,
  
	\begin{align*}
    \ol\bD(\ol f(\ell_{(x_n)}),\widetilde{f}(\ell_{(x_n)}))
    &=\ol\bD\left(\ol f\left(\lim_{n\to\infty}\iota_\bC(x_n)\right),\widetilde{f}\left(\lim_{n\to\infty}\iota_\bC(x_n)\right)\right) &(\text{\autoref{cor:lim_iota_cauchy_seq}}) \\
    &=\ol\bD\left(\lim_{n\to\infty}\ol f(\iota_\bC(x_n)),\lim_{n\to\infty}\widetilde{f}(\iota_\bC(x_n))\right) &(\text{\autoref{prop:R-functors-preserve-everything}}) \\
    &=\ol\bD\left(\lim_{n\to\infty}\iota_\bD(f(x_n)),\lim_{n\to\infty}\iota_\bD(f(x_n))\right)=0.&\qedhere
	\end{align*}
\end{proof}

\begin{lemma}\label{prop:R_pastoral_equiv}
  Given an ${\bR_+}$-functor $f:\bC\to\bD$ between symmetric
  $\mathbb{R}_+$-categories, the following are equivalent.
	\begin{enumerate}[label=(\arabic*)]
    \item $f$ is fully faithful and dense. 
    \item The $\mathbb{R}_+$-functor $\overline f:\overline\bC\to\overline\bD$
    (as defined in \autoref{prop:induced_R-functor_between_completions}) is
    fully faithful and strictly surjective.
    \item The $\mathbb{R}_+$-functor $\overline f:\overline\bC\to\overline\bD$
    is an ${\bR_+}$-equivalence of categories.
	\end{enumerate}
\end{lemma}
\begin{proof}
  We first show (1)$\implies$(2). Suppose $f:\bC\to\bD$ is fully faithful and
  dense. Let $(x_n)$ and $(y_n)$ be Cauchy sequences in $\bC$. We have
  \[
    \overline{\mathbb{D}}(\ol f(\ell_{(x_n)}),\overline{f}(\ell_{(y_n)}))
    =\lim_{n\to\infty}\bD(f(x_n),f(y_n))
    =\lim_{n\to\infty}\bC(x_n,y_n)
    =\ol\bC(\ell_{(x_n)},\ell_{(y_n)}),
  \]
  where the first and last equality follow by
  \autoref{prop:Lmetric-on-cauchy-completion}. Hence, $\ol f$ is fully
  faithful. Finally, we claim that $\ol f$ is surjective. Let $\ell_{(y_n)}$ be
  a presheaf in $\ol\bD$ determined by some Cauchy sequence $(y_n)$ in
  $\bD$. We want to show that there exists a Cauchy sequence $(x_n)$ in $\bC$
  such that $\ol f(\ell_{(x_n)})=\ell_{(y_n)}$. Since $f$ is dense, for each
  $n\in\bN$, there exists a sequence $(x_{n,k})_{k\in\bN}$ in $\bC$ such that
  $y_n$ is a limit of $(f(x_{n,k}))_{k\in\bN}$. Then by
  \autoref{lemma:diagonal-is-equivalent}, there is a sequence $(x_n)$ in $\bC$
  such that the sequence $(f(x_n))$ is equivalent to $(y_n)$. Then because $f$
  is fully faithful and $(f(x_n))$ is a Cauchy sequence, $(x_n)$ is a Cauchy
  sequence in $\bC$, by \autoref{prop:R-functors-preserve-everything}. Then by
  \autoref{prop:equiv-cauchy-same-presheaf}
  \[
    \ol f(\ell_{(x_n)})
    =\ell_{(f(x_n))}
    =\ell_{(y_n)}.
  \]
  Thus, $\overline{f}$ is surjective.

  That (2)$\implies$(3) is clear.

  Finally, we show that (3)$\implies$(1). Suppose $\ol f$ is an equivalence of
  categories. Since both $\ol f$ and $\iota_\bC$ are fully faithful, we know
  that $\iota_\bD\circ f=\ol f\circ\iota_\bC$ is fully faithful. Furthermore,
  since $\iota_\bD$ is fully faithful, we have for all $x,y\in\bC$ that
  \[\bD(f(x),f(y))=\ol\bD(\iota_\bD(f(x)),\iota_\bD(f(y)))=\bC(x,y).\] Hence,
  $f$ is likewise fully faithful. Let $d\in\bD$. Since $\ol f$ is
  essentially surjective, there exists some Cauchy sequence $(x_n)$ in
  $\mathbb{C}$ such that $\ol f(\ell_{(x_n)})\cong\iota_\bD(d)$. In particular,
  this means:
  \[
    \ol\bD(\ol f(\ell_{(x_n)}),\iota_\bD(d))
    =\ol\bD(\ell_{f((x_n))},\ell_{\wh d})
    =\lim_{n\to\infty}\bD(f(x_n),d)
    =0,
  \]
  where the second equality follows by
  \autoref{prop:Lmetric-on-cauchy-completion}. Hence, $d$ is a limit of
  $(f(x_n))$, and so $f$ is dense.
\end{proof}

The following proofs concern the $\bR_+$-category $\Seq$ defined in \autoref{def:Seq}.

\begin{lemma}\label{lemma:convergent-Cauchy-const-or-inf}
	Every Cauchy sequence in $\Seq$ is either eventually constant or equivalent to $(n)_{n\in\NN}$. 
\end{lemma} 
\begin{proof}
	Let $(x_n)$ be a Cauchy sequence in $\Seq$ that is not eventually constant. We first claim that $(x_n)$ grows arbitrarily large. Let $n\in\bN$. Since $(x_n)$ is Cauchy, there exists $N(n)\in\bN$ such that
	\[m,k\geq N(n)\implies\Seq(x_m,x_k)<\frac1{2^n}.\]
	Now we claim $x_m\geq n$ for all $m\geq N(n)$. To see this, let $m\geq N(n)$. Then since $(x_n)$ isn't eventually constant, there exists $k\geq m$ such that $x_k\neq x_m$. Then we have
	\[\frac1{2^{\min(x_k,x_m)}}\leq\sum_{i=\min(x_k,x_m)}^{\max(x_k,x_m)-1}\frac1{2^i}=\Seq(x_k,x_m)<\frac1{2^n}\implies n<\min(x_k,x_m)\leq x_m,\]
	as desired. 

	Now, to see $(x_n)\sim(n)$, let $\vare>0$, and choose $n_\vare\in\bN$ such that $\sum_{i=n_\vare}^\infty\frac1{2^i}<\vare$.\footnote{Such an $n_\vare$ exists because $\sum_{i=1}^\infty\frac1{2^i}$ converges to $1$.} Then for all $n\geq\max(N(n_\vare),n_\vare)$, we have $x_n,n\geq n_\vare$, so that
	\[\Seq(x_n,n)=\sum_{i=\min(x_n,n)}^{\max(x_n,n)-1}\frac1{2^i}<\sum_{i=n_\vare}^\infty\frac1{2^i}<\vare.\]
  Thus we've shown $\lim_{n\to\infty}\Seq(x_n,n)=0$, and $\Seq$ is symmetric,
  so $(x_n)$ and $(n)$ are equivalent Cauchy sequences, as desired.
\end{proof}

By the above lemma and \autoref{prop:equivalent_Cauchy_characterization}, we
have the following result.

\begin{corollary}\label{cor:description_of_seq_bar}
	Every element of $\ol\Seq$ is equal to either $\ell_{\wh m}$ for some $m\in\bN$ or $\ell_{(n)}$.
\end{corollary}

\begin{lemma}\label{prop:Seq_pick_out_subsequence}
  Let $\mathbb{C}$ be a small $\mathbb{R}_+$-category. Given a Cauchy sequence
  $(x_n)$ in $\bC$, there exists an ${\bR_+}$-functor $f:\Seq\to\bC$ such that
  the Cauchy sequence $(f(n))_{n\in\bN}$ is a subsequence of $(x_n)$.
\end{lemma}
\begin{proof}
  Since $(x_n)$ is Cauchy, for each $n\in\mathbb{N}$ there exists some
  $N(n)\in\mathbb{N}$ such that
  $\mathbb{C}(x_m,x_k),\mathbb{C}(x_k,x_m)\leq\frac{1}{2^n}$ whenever $m,k\geq
  N(n)$. Moreover, we can take the assignment $n\mapsto N(n)$ to be strictly
  increasing, by inductively redefining
  \[
    N(1)=N(1)
    \qquad\text{and}\qquad
    N(n+1):=\max(N(n)+1,N(n+1)).
  \]
  Now define $f(n)=x_{N(n)}$, so clearly $(f(n))$ is a subsequence of
  $(x_n)$. Moreover, $f$ is an $\mathbb{R}_+$-functor: Clearly
  $\mathbb{C}(f(n),f(n+1))\leq\frac{1}{2^n}$ for all $n\in\mathbb{N}$, and then
  for $n<m$, by the triangle inequality we get
  \[
    \mathbb{C}(f(n),f(m))
    \leq\sum_{i=1}^{m-1}\mathbb{C}(f(i),f(i+1))
    \leq\sum_{i=n}^{m-1}\frac{1}{2^i}
    =\Seq(n,m)
  \]
  (and a similar argument yields $(\mathbb{C}(f(m),f(n))\leq\Seq(n,m)
  =\Seq(m,n)$).
\end{proof}

\begin{lemma}\label{lemma:convergent_Cauchy_factors_through_seq_bar}
  Given a symmetric ${\bR_+}$-category $\bC$, an ${\bR_+}$-functor
  $f:\Seq\to\bC$ factors through the inclusion
  $\iota_\Seq:\Seq\into\overline\Seq$ (i.e., there exists some dashed map $g$
  that makes the following diagram commute)
	\[\begin{tikzcd}
		\Seq && \bC \\
		& \overline\Seq
		\arrow["f", from=1-1, to=1-3]
		\arrow["{\iota_\Seq}"', hook, from=1-1, to=2-2]
		\arrow["{g}"', dashed, from=2-2, to=1-3]
	\end{tikzcd}\]
  if and only if the sequence $\{f(n)\}$ has a limit in $\bC$, in which case
  the dashed map sends the presheaf $\ell_{(n)}$ to a limit of
  $\{f(n)\}$. Furthermore, there exists a unique such $g$ sending
  $\ell_{\{m\}}$ to $c$ for each distinct limit $c$ of $\{f(n)\}$.
\end{lemma}
\begin{proof}
  Let $\bC$ be a symmetric $\bR_+$-category and let $f:\Seq\to\bC$ be an
  $\bR_+$-functor. First, suppose that $f$ factors through $\ol\Seq$ as
  $g\circ\iota_\Seq$. Then since $\ell_{(n)}$ is a limit of $(\iota_\Seq(n))$
  (\autoref{cor:lim_iota_cauchy_seq}), $g(\ell_{(n)})$ must be a limit of
  $(g(\iota_\Seq(n)))=(f(n))$ by \autoref{prop:R-functors-preserve-everything},
  so that indeed $(f(n))$ has a limit.
	
  Conversely, suppose that $(f(n))$ has a limit $c\in\bC$. By
  \autoref{cor:description_of_seq_bar}, in order to define an $\bR_+$-functor
  $g:\ol\Seq\to\bC$, it suffices to define $g(\ell_{(n)})$ and $g(\ell_{\wh
  m})$ for all $m\in\bN$, and show that $g$ is an $\mathbb{R}_+$-functor. It is
  then straightforward to see that we \emph{must} define
  \[
    g(\ell_{(n)})=c
    \qquad\text{and}\qquad
    g(\ell_{\wh m})=f(m).
  \]
  in order to get a map $\ol\Seq\to\bC$ that satisfies $f=g\circ\iota_\Seq$ and
  sends $\ell_{(n)}\mapsto c$. It remains to show that $g$ is an
  $\mathbb{R}_+$-functor. Let $\ell_{(x_n)}$ and $\ell_{(y_n)}$ be two
  presheaves in $\ol\Seq$ associated to Cauchy sequences $(x_n)$ and $(y_n)$ in
  $\Seq$. We wish to show that
  \[
    \ol\Seq(\ell_{(x_n)},\ell_{(y_n)})
    \geq\bC(g(\ell_{(x_n)}),g(\ell_{(y_n)})).
  \]
	In the case that $\ell_{(x_n)}=\ell_{(y_n)}$, this is clearly true. In the case that $\ell_{(x_n)}=\ell_{\wh i}$ and $\ell_{(y_n)}=\ell_{\wh j}$ for some distinct $i,j\in\Seq$, we have that:
	\[
	\begin{aligned}
	    	\ol\Seq(\ell_{(x_n)},\ell_{(y_n)})&=\ol\Seq(\iota_\Seq(i),\iota_\Seq(j))\\&=\Seq(i,j) \\
		&\geq\bC(f(i),f(j))\\
  &=\bC(g(\iota_\Seq(i)),g(\iota_\Seq(j)))\\
  &=\bC(g(\ell_{(x_n)}),g(\ell_{(y_n)})).
	\end{aligned}
	\]
	Now, suppose $\ell_{(x_n)}=\ell_{(n)}$ and $\ell_{(y_n)}=\ell_{\wh m}$ for some $m\in\bN$. Then
	\begin{align*}
		\ol\Seq(\ell_{(x_n)},\ell_{(y_n)})&=\lim_{n\to\infty}\Seq(n,m) &&(\text{\autoref{prop:Lmetric-on-cauchy-completion}}) \\
		&\geq\lim_{n\to\infty}\bC(f(n),f(m)) &&(\text{\autoref{prop:lim_of_comparables}}) \\
		&=\bC(c,f(m)) &&(\text{\autoref{lemma:limits_can_be_moved_out_of_distance_function}}) \\
		&=\bC(g(\ell_{(x_n)}),g(\ell_{(y_n)})).
	\end{align*}
  By \autoref{cor:description_of_seq_bar} and symmetry, we have covered all possible cases, so that indeed $g$ is an $\mathbb{R}_+$-functor.
\end{proof}

\section*{Acknowledgments}
This article was first written during the 2022 UVA Topology Research Experience
for Undergraduates (REU) to share our research with the other students. Thank
you to the organizers who made this program possible and special thanks to our
wonderful mentor Walker Stern, and graduate students Tanner Carawan and Miika
Tuominen who offered their advice and insight throughout the
project. Additional thanks to the NSF Research Training Grant (RTG) in Geometry
and Topology (DMS-1839968) that partly supported the REU. We are also grateful
to Javier Guti\'errez, for pointing out connections with existing work and
suggesting \autoref{rmk:localizations} to us. Finally, we extend our most
sincere gratitude to the anonymous referee, whose detailed and clear comments
greatly improved this work. 

%
%
%
%

\bibliographystyle{spmpsci}
\bibliography{refs} 

\end{document}